\documentclass[hidelinks,onefignum,onetabnum]{siamart251216}
\usepackage[utf8]{inputenc}
\usepackage[T1]{fontenc}
\usepackage{amssymb,mathtools}
\usepackage{stmaryrd}
\SetSymbolFont{stmry}{bold}{U}{stmry}{m}{n}
\usepackage{bm}
\usepackage{mathrsfs}
\usepackage{graphicx}
\usepackage{booktabs}
\usepackage{enumitem}
\usepackage{xcolor}
\hypersetup{colorlinks=true,linkcolor=blue,citecolor=blue,urlcolor=blue}
\usepackage{tikz}
\usepackage{tikz-cd}
\usetikzlibrary{calc,arrows.meta,decorations.pathreplacing}
\usepackage{tikz-3dplot}
\usetikzlibrary{calc}

\usepackage{tcolorbox}
\tcbuselibrary{breakable,skins}
\definecolor{errorrevisionyellow}{RGB}{255,246,170}
\newtcolorbox{errorrevision}{enhanced,breakable,
  colback=errorrevisionyellow,colframe=errorrevisionyellow,
  boxrule=0pt,arc=0pt,outer arc=0pt,
  left=2pt,right=2pt,top=2pt,bottom=2pt,boxsep=0pt,
  before skip=5pt,after skip=5pt,pad at break*=1pt}

\newsiamthm{assumption}{Assumption}
\newsiamthm{problem}{Problem}
\newsiamremark{remark}{Remark}

\newcommand{\Fht}{  \Phi_{\tilde{T}}  }
\newcommand{\Fhk}{  \Phi_T  }
\newcommand{\Fhr}{  \Phi_{\check{T}}  }
\newcommand{\Ftk}{  \Psi_T  }

\newcommand{\Fkr}{  \Theta_{\check{T}} }

\newcommand{\Ftkh}{  \Psi_h  }
\newcommand{\Fkrh}{  \Theta_h }

\newcommand{\FOhkpm}{  \Phi_{T^\pm}  }

\newcommand{\FOhrpm}{ \Phi_{\rT^\pm} }

\newcommand{\Fhkpm}{  {\Phi_{T}^\pm}  }

\newcommand{\Fhrpm}{ {\Phi_{\rT}^\pm} }
\newcommand{\FFhkp}{\phi_{F}^+}
\newcommand{\FFhkm}{\phi_{F}^-}
\newcommand{\FFhkpm}{\phi_{F}^\pm}
\newcommand{\FFhrp}{ \phi_{\rF}^+}
\newcommand{\FFhrm}{ \phi_{\rF}^-}
\newcommand{\FFhrpm}{ \phi_{\rF}^\pm}

\newcommand{\Phks}{ \mathcal{P}_{\Fhk}^\Sigma }
\newcommand{\Phkw}{ \mathcal{P}_{\Fhk}^W }
\newcommand{\Phkv}{ \mathcal{P}_{\Fhk}^V }
\newcommand{\Phkq}{ \mathcal{P}_{\Fhk}^Q }
\newcommand{\Phtq}{ \mathcal{P}_{\Fht}^Q }

\newcommand{\Ptkv}{ \mathcal{P}_{\Ftk}^V }
\newcommand{\Ptkq}{ \mathcal{P}_{\Ftk}^Q }

\newcommand{\Phtv}{ \mathcal{P}_{\Fht}^V }

\newcommand{\Phkxpm}{ \mathcal{P}_{\Fhkpm}^X}

\newcommand{\PFhkxp}{ \mathcal{P}_{\FFhkp}^{X} }
\newcommand{\PFhkxm}{ \mathcal{P}_{\FFhkm}^{X} }
\newcommand{\PFhkxpm}{ \mathcal{P}_{\FFhkpm}^{X} }

\newcommand{\Rhkw}{ \mathcal{R}_{\Fhk}^W }
\newcommand{\Rhkv}{ \mathcal{R}_{\Fhk}^V }
\newcommand{\Rhkq}{ \mathcal{R}_{\Fhk}^Q }
\newcommand{\Rhrs}{ \mathcal{R}_{\Fhr}^\Sigma}
\newcommand{\Rhrw}{ \mathcal{R}_{\Fhr}^W }

\newcommand{\Rtkv}{ \mathcal{R}_{\Ftk}^V}
\newcommand{\Rtkq}{ \mathcal{R}_{\Ftk}^Q }

\newcommand{\Rkrw}{ \mathcal{R}_{\Fkr}^W}

\newcommand{\Rhrxpm}{ \mathcal{R}_{\Fhrpm}^X}

\newcommand{\RFhrxp}{ \mathcal{R}_{\FFhrp}^{X}}
\newcommand{\RFhrxm}{ \mathcal{R}_{\FFhrm}^{X}}
\newcommand{\RFhrxpm}{ \mathcal{R}_{\FFhrpm}^{X}}

\newcommand{\FFhh}{ \gamma_F }

\newcommand{\PFhhx}{ \mathcal P_{\FFhh}^{X}}
\newcommand{\RFhhx}{ \mathcal R_{\FFhh}^{X}}
\newcommand{\RFhhs}{ \mathcal R_{\FFhh}^{\Sigma}}
\newcommand{\RFhhw}{ \mathcal R_{\FFhh}^{W}}
\newcommand{\rT}{\check{T}}
\newcommand{\rF}{\check{F}}

\newcommand{\pis}{\Pi^\Sigma_h}
\newcommand{\pist}{\Pi_T^\Sigma}
\newcommand{\hpis}{\Pi^\Sigma_{\hT}}

\newcommand{\pishF}{\mathcal{I}^\Sigma_{\hF}}
\newcommand{\pishFp}{\mathcal{I}^\Sigma_{\hF^+}}
\newcommand{\pishFm}{\mathcal{I}^\Sigma_{\hF^-}}

\newcommand{\piw}{\bPi^W_h}
\newcommand{\piwt}{\bPi_T^W}
\newcommand{\hpiw}{\bPi^W_{\hT}}

\newcommand{\piwhF}{\boldsymbol{\mathcal{I}}^W_{\hF}}
\newcommand{\piwhFp}{\boldsymbol{\mathcal{I}}^W_{\hF^+}}
\newcommand{\piwhFm}{\boldsymbol{\mathcal{I}}^W_{\hF^-}}

\newcommand{\pix}{\Pi^X_h}

\newcommand{\pixhT}{\Pi^X_{\hT}}
\newcommand{\pixtp}{\Pi^X_{T^+}}
\newcommand{\pixtm}{\Pi^X_{T^-}}
\newcommand{\pixtpm}{\Pi^X_{T^\pm}}

\newcommand{\pixhF}{\mathcal{I}^X_{\hF}}
\newcommand{\pixhFp}{\mathcal{I}^X_{\hF^+}}
\newcommand{\pixhFm}{\mathcal{I}^X_{\hF^-}}
\newcommand{\pixhFpm}{\mathcal{I}^X_{\hF^\pm}}

\newcommand{\bzero}{\bm{0}}
\newcommand{\jump}[1]{\llbracket #1 \rrbracket}
\newcommand{\aver}[1]{\{ #1 \}}

\newcommand{\Dtau}{D_\tau}

\newcommand{\bu}{\boldsymbol{u}}
\newcommand{\bH}{\boldsymbol{H}}
\newcommand{\bC}{\boldsymbol{C}}
\newcommand{\bv}{\boldsymbol{v}}
\newcommand{\bz}{\boldsymbol{z}}
\newcommand{\br}{\boldsymbol{r}}
\newcommand{\bV}{\boldsymbol{V}}

\newcommand{\bx}{x}

\newcommand{\bP}{\boldsymbol{P}}

\newcommand{\bw}{\boldsymbol{w}}
\newcommand{\btau}{\boldsymbol{\tau}}
\newcommand{\bt}{\boldsymbol{t}}

\newcommand{\bn}{\boldsymbol{n}}
\newcommand{\hbn}{\hat{\bn}}
\newcommand{\bL}{\boldsymbol{L}}
\newcommand{\bZ}{\boldsymbol{Z}}

\newcommand{\bft}{\boldsymbol{f}}

\newcommand{\bg}{\boldsymbol{g}}
\newcommand{\bb}{\boldsymbol{b}}

\newcommand{\bI}{\boldsymbol{I}}
\newcommand{\bE}{\boldsymbol{E}}
\newcommand{\bW}{\boldsymbol{W}}

\newcommand{\bIT}{\bI_T}

\newcommand{\dx}{\,\rd x}
\newcommand{\hdx}{\,\rd \hat{x}}

\newcommand{\ds}{\,\rd s}
\newcommand{\hds}{\,\rd\hat{s}}
\newcommand{\tds}{\,\rd\tilde{s}}
\newcommand{\curl}{\operatorname{curl}}
\newcommand{\Div}{\operatorname{div}}
\newcommand{\rot}{\operatorname{rot}}
\newcommand{\grad}{\operatorname{grad}}

\newcommand{\tDiv}{\tilde{\operatorname{div}}}

\newcommand{\Eh}{\boldsymbol{E}_h}

\newcommand{\Sh}{\boldsymbol{S}_h}
\newcommand{\ST}{\boldsymbol{S}_T}
\newcommand{\Gh}{\Upsilon_h}

\newcommand{\bPi}{\boldsymbol{\Pi}}

\newcommand{\calTh}{\mathcal{T}_h}

\newcommand{\calFh}{\mathcal{F}_h}

\newcommand{\calFhI}{\mathcal{F}_h^i}
\newcommand{\calFhIc}{\mathcal{F}_h^{i}}

\newcommand{\NT}{\mathcal{N}(T)}

\newcommand{\tx}{\tilde{x}}
\newcommand{\ta}{\tilde{a}}
\newcommand{\tT}{\tilde{T}}
\newcommand{\tF}{\tilde{F}}
\newcommand{\tq}{\tilde{q}}
\newcommand{\tQ}{\tilde{Q}}
\newcommand{\tY}{\tilde{Y}}
\newcommand{\tbn}{\tilde{\bn}}
\newcommand{\tbv}{\tilde{\bv}}
\newcommand{\tbw}{\tilde{\bw}}
\newcommand{\tbV}{\tilde{\bV}}

\newcommand{\tOmega}{\tilde{\Omega}}
\newcommand{\tcalTh}{\tilde{\mathcal{T}}_h}

\newcommand{\tNT}{\mathcal{N}(\tT)}

\newcommand{\ha}{\hat{a}}
\newcommand{\hx}{\hat{x}}
\newcommand{\hT}{\hat{T}}
\newcommand{\hF}{\hat{F}}
\newcommand{\hK}{\hat{K}}
\newcommand{\hq}{\hat{q}}
\newcommand{\hQ}{\hat{Q}}
\newcommand{\hbv}{\hat{\bv}}

\newcommand{\hbV}{\hat{\bV}}

\newcommand{\hNT}{\mathcal{N}(\hT)}

\newcommand{\hbw}{\hat{\bw}}
\newcommand{\hbW}{\hat{\bW}}
\newcommand{\hnab}{\hat{\nabla}}

\newcommand{\rd}{\mathrm{d}}
\newcommand{\trans}{\mathsf{T}}

\newcommand{\TrhF}{\operatorname{tr}_{\hat F}}

\newcommand{\TrxF}{\operatorname{tr}^{X}_{F}}
\newcommand{\TrxrF}{\operatorname{tr}^{X}_{\rF}}
\newcommand{\TrxhF}{\operatorname{tr}^{X}_{\hF}}

\newcommand{\TrxhFpm}{\operatorname{tr}^{X}_{\hF^\pm}}
\newcommand{\TrsFz}{\operatorname{tr}^{\Sigma}_{F_0}}

\newcommand{\TrxFz}{\operatorname{tr}^{X}_{F_0}}
\newcommand{\TrxFo}{\operatorname{tr}^{X}_{F_1}}
\newcommand{\TrxFt}{\operatorname{tr}^{X}_{F_2}}
\newcommand{\TrxFi}{\operatorname{tr}^{X}_{F_i}}
\newcommand{\TrsF}{\operatorname{tr}^{\Sigma}_{F}}
\newcommand{\TrshF}{\operatorname{tr}^{\Sigma}_{\hF}}

\newcommand{\TrwF}{\operatorname{tr}^{W}_{F}}

\newcommand{\TrwFo}{\operatorname{tr}^{W}_{F_1}}

\newcommand{\TrwhF}{\operatorname{tr}^{W}_{\hF}}

\headers{Pressure-Robust Finite Elements for Stokes}{Wei Chen and Zhen Liu}

\title{
Pressure-Robust Finite Elements for the Stokes Problem
on Three-Dimensional Curved Domains
}

\author{Wei Chen\thanks{LMAM and School of Mathematical Sciences, Peking University, Beijing 100871, P. R. China;
Chongqing Research Institute of Big Data, Peking University, Chongqing 401329, P. R. China
(\email{2406397052@pku.edu.cn}).
The first author was supported by the China Postdoctoral Science Foundation, Grant No. 2025M783065.}
\and Zhen Liu\thanks{Corresponding author.
Institut für Mathematik, Friedrich-Schiller-Universität Jena, 07743 Jena, Germany;
Chongqing Research Institute of Big Data, Peking University, Chongqing 401329, P. R. China
(\email{z.liu@uni-jena.de}).
The second author was supported by the Sino-German (CSC-DAAD) Postdoc Scholarship Program, ID 202406010516.}}

\begin{document}
\maketitle

\begin{abstract}
This paper develops a divergence-free, inf-sup stable, optimally convergent and pressure-robust finite element method for the three-dimensional Stokes problem on curved domains. The geometry is approximated by an isoparametric tetrahedral mesh, while the velocity space is obtained from Scott--Vogelius spaces on an Alfeld split by the Piola transform. The discrete inf-sup condition is proved using suitable face bubble functions. The insufficient accuracy of quadrature rules on curved triangular interfaces leads to a consistency error that may cause suboptimal convergence. The remedy is to introduce a suitable consistency correction without stabilization terms. Moreover, commuting operators on curved domains are constructed from local commuting interpolants and shown to be globally conforming. The operators are used to approximate the load and obtain a pressure-robust discretization. Numerical examples are provided to validate the theoretical results.
\end{abstract}



\section{Introduction}
\label{sec:introduction}

Let $\Omega \subset \mathbb{R}^3$ be a bounded domain with sufficiently smooth boundary $\partial\Omega$, and let $\bV:=\bH_0^1(\Omega)$ and $Q:=L_0^2(\Omega)$. 
The weak formulation of the Stokes equation reads: Given $\bft \in \bL^2(\Omega)$, find $(\bu,p)\in\bV\times Q$ such that
\begin{equation}
\label{eq:continuous-mixed}
\begin{aligned}
\nu(\nabla\bu,\nabla\bv)_\Omega-(p,\Div\bv)_\Omega
    &=(\bft,\bv)_\Omega &&\text{ for all } \bv\in\bV,\\
(\Div\bu,q)_\Omega&=0 &&\text{ for all } q\in Q,
\end{aligned}
\end{equation}
where $\nu>0$ denotes the kinematic viscosity, $\bu$ the fluid velocity, $p$ the pressure, $\bft$ the body force and $(\bullet,\bullet)_\Omega$ the $L^2$ inner product in $\Omega$. 

Finite element methods that produce divergence-free velocity approximations on polyhedral domains have several attractive properties. They preserve mass conservation at the discrete level and naturally decouple the velocity error from the pressure approximation. This decoupling, known as pressure robustness, is particularly important for flows with large pressure gradients or small viscosity. Consequently, the development of divergence-free and pressure-robust finite element methods has become an active area of research; see, for example, \cite{john2017on,scott1985norm,guzman2014conforming3D,guzman2018inf,fabien2022low,zhang2005a,beirao2017divergence} and the references therein. The situation is more delicate on curved domains. When high-order finite element spaces are used on a straight-sided approximation of the curved domain, geometric errors may limit the overall accuracy. Isoparametric methods mitigate this difficulty by employing polynomial mappings to construct curved computational elements. However, a standard isoparametric transformation by composition generally fails to preserve the divergence-free and pressure-robust properties. Neilan and Otus \cite{neilan2021divergence} addressed this difficulty in two dimensions by defining the velocity space through the Piola transform. Their construction is based on the lowest-order Scott--Vogelius pair on Clough--Tocher splits and yields an inf-sup stable, divergence-free, and optimally convergent method on curved domains. The resulting velocity space is globally $H(\Div)$-conforming and possesses suitable weak continuity properties, although it is generally not $H^{1}$-conforming near the curved boundary. This framework has subsequently been extended to boundary-correction method \cite{liu2023a}, $H^{1}$-conforming construction \cite{neilan2023a}, arbitrary polynomial degree \cite{durst2024a}, and nonconforming Fortin--Soulie method \cite{chen2026pressure}.
Based on an interior-penalty discontinuous Galerkin (IPDG) formulation \cite{arnold2001unified,schotzau2002mixed}, a scheme employing Brezzi--Douglas--Marini velocity spaces was proposed by Li et al. \cite{li2025divergence} for three-dimensional curved domains. The resulting method produces divergence-free velocity approximations and achieves high-order optimal convergence.

This paper extends the two-dimensional method in \cite{neilan2021divergence} to construct a divergence-free, uniformly inf-sup stable, optimally convergent, and pressure-robust element for three dimensions. The construction starts from the inf-sup stable Scott--Vogelius pair with the polynomial degree $k\geq3$ on the Alfeld split of a reference tetrahedron \cite{guzman2018inf}. The velocity space is obtained from the Scott--Vogelius space on an Alfeld-split reference tetrahedron by the Piola transform, whereas the pressure space is mapped by composition. The global velocity space is assembled by matching the associated Lagrange degrees of freedom across common macroelement faces. Although the construction follows from \cite{neilan2021divergence}, there are indeed some difficulties that need to be addressed specifically in the three-dimensional case.

Unlike the two-dimensional case, interior faces of a curved tetrahedral mesh need not be planar, and the normal traces of the discrete velocity on these faces are generally nonpolynomial. The arguments for $H(\operatorname{div})$-conformity and inf-sup stability in \cite[Theorem~4.2 and Lemma~4.3]{neilan2021divergence} do not apply directly. For $H(\operatorname{div})$-conformity, the required polynomial structure is recovered by pulling back the normal trace and scaling it by the surface Jacobian. To establish inf-sup stability, specialized face bubble functions are constructed on the curved interior faces to control the piecewise constant pressure component. Together with the local inf-sup stability of the Scott--Vogelius pair, this establishes a discrete inf-sup condition that is uniform with respect to the mesh size for sufficiently fine meshes.

The lack of sufficient weak continuity of the discrete velocity space may cause suboptimal convergence. In two dimensions, the Gauss--Lobatto interpolation nodes on common edges provide sufficient quadrature accuracy to control the interface consistency error \cite{durst2024a}. No analogous quadrature rule based on the Lagrange interpolation nodes is available on triangular faces in three dimensions \cite{xu2011on}. Consequently, the direct three-dimensional extension of the method in \cite{neilan2021divergence} yields suboptimal convergence rates. The IPDG method addresses this issue by introducing consistency terms, but also requires a penalty term involving a sufficiently large stabilization parameter to ensure coercivity \cite{li2025divergence}. In contrast, this paper proposes the following velocity bilinear form without penalty terms:
$$
a_h(\bu_h,\bv_h) =\sum_{T\in\calTh}(\nabla\bu_h,\nabla\bv_h)_T - \sum_{F \in \calFhI} \int_F \aver{\partial_{\bn_F}\bu_h}  \cdot \jump{\bv_h} +  \aver{\partial_{\bn_F}\bv_h}  \cdot \jump{\bu_h}\rd s,
$$
where $\calTh$ is a curved macro-tetrahedral mesh, $\calFhI$ denotes the set of interior faces, $\bn_F$ is a fixed unit normal to $F$, and $\jump{\bullet}$ and $\aver{\bullet}$ denote the jump and average of the traces across $F$, respectively.
Since the discrete velocity functions are continuous at the Lagrange nodes on each interior face, one can estimate the scaled jump term by the broken $H^1$-seminorm with an additional factor of the mesh size $h$. This, combined with an inverse trace inequality, proves that $a_h(\bullet,\bullet)$ is coercive for sufficiently small $h$. 

Since the computational domain $\Omega_h$ does not coincide with the physical domain $\Omega$, another challenge is to construct a computable load $\bft_h$ on $\Omega_h$ such that the resulting scheme is pressure-robust.  Neilan and Otus \cite[Corollary 6.2]{neilan2021divergence} showed that this can be done by applying a commuting operator to the source term $\bft$ defined on $\Omega$. This construction is carried out directly on curved elements in two dimensions, and its extension to three dimensions and high-order cases is not straightforward. In a unified way, this paper develops commuting operators in three dimensions for all admissible polynomial degrees. On each element, the local operator is defined as the composition of three operators: the pullback from the physical element to the reference element, the interpolation operator defined on the reference Alfeld tetrahedron introduced in \cite{fu2020exact}, and the covariant pushforward from the reference element to the curved computational element. The local commuting property follows from the chain rule and the commuting result on the reference element \cite[Theorem 4.20]{fu2020exact}. An additional issue is the global conformity of the interpolated functions on curved domains. This is established through trace compatibility of the local interpolation operators and their invariance under the corresponding face transformations. The resulting construction of commuting operators on curved domains is not specific to the present setting and provides a general principle that extends to other reference commuting pairs.

The rest of the paper is organized as follows: Section~\ref{sec:preliminaries} introduces the curved triangulation, the maps and transformations. Section~\ref{sec:global-spaces} constructs the local and global curved Scott--Vogelius spaces and proves their $H(\Div)$-conformity and inf-sup stability. Section~\ref{sec:discrete-method} presents the discrete method and establishes the well-posedness. The pressure robustness is developed in Section~\ref{sec:pressure-robust}. Numerical examples are provided in Section~\ref{sec:numerics}.

\section{Preliminaries}
\label{sec:preliminaries}

\subsection{Notation}
Given a bounded domain $G$ with the boundary $\partial G$, the usual Sobolev spaces $W^{m,p}(G)$ with norm $\| \bullet\|_{W^{m,p}(G)}$ and semi-norm $|\bullet |_{W^{m,p}(G)}$ are used~\cite{adams2003sobolev}. Let $H^m(G) = W^{m,2}(G)$ and $L^p(G) = W^{0,p}(G)$. Furthermore, $H_0^1(G)$ is the subspace of $H^1(G)$ consisting of functions with vanishing traces on $\partial G$, and $L^2_0(G)$ denotes the space of $L^2(G)$ functions with zero mean over $G$. Denote the space of $m$-times continuously differentiable functions on $G$ by $C^m(G)$ and the space of all polynomials on $G$ with degree less than or equal to $k$ by $P_k(G)$. Corresponding vector-valued functions and spaces are denoted by boldface letters, e.g., $\bv \in \bP_k(G)$. Let $(\bullet, \bullet)_G$ denote the $L^2$ inner product on $G$ between scalars, vectors, or tensors.

Throughout this paper, the notation $A \lesssim B$ denotes $A\le C B$, where $C>0$ is a constant independent of the mesh size $h$ and the kinematic viscosity $\nu$. Similarly, $A \approx B$ signifies $A\lesssim B$ and $B\lesssim A$. The differential operator $\nabla$ is understood elementwise with respect to meshes. For a regular mapping $\Phi:~\mathbb{R}^3 \rightarrow \mathbb{R}^3$, denote its Jacobian by $D\Phi$ and set $J_\Phi = \det D\Phi$. 

\subsection{Curved meshes and mappings}
Throughout the paper, assume that $\Omega\subset\mathbb{R}^3$ is a bounded domain whose boundary $\partial\Omega$ is sufficiently smooth and can be represented by a finite number of local charts. The polynomial degree $k \geq 3$ is fixed throughout.
Following the standard isoparametric framework (see, e.g., \cite{lenoir1986optimal,bernardi1989optimal,brenner2008the}), this section constructs a curved tetrahedral mesh $\calTh$ of degree $k$.

The construction begins with a shape-regular affine tetrahedral mesh $\tcalTh$ whose interior vertices lie in $\Omega$ and boundary vertices lie on $\partial\Omega$. Assume that each $\tT\in\tcalTh$ has at most three vertices on $\partial\Omega$.
The polyhedral domain $\tOmega_h\coloneqq{\rm int} \left(\cup_{\tT \in \tcalTh} \overline{\tT}\right)$ provides an $\mathcal{O}(h^2)$ approximation of~$\Omega$. Here $h\coloneqq\max_{\tT\in\tcalTh}\operatorname{diam}(\tT)$.
Let $\hT$ be the reference tetrahedron with vertices $(0,0,0)^\trans$, $(1,0,0)^{\trans}$, $(0,1,0)^\trans$ and $(0,0,1)^\trans$, and let $\Fht:~\hT\to \tT$ be an affine bijection.

For $\tT\in\tcalTh$ with a boundary edge or face, construct a polynomial map $\Fhk:\hT\to\mathbb{R}^3$ recursively from the affine map~$\Fht$ as in \cite[Equation~(22)]{lenoir1986optimal}, which is of degree at most $k$ and maps $\hT$ onto a curvilinear tetrahedron~$T$. For $\tT\in\tcalTh$ with no boundary edge or face, set $\Fhk\coloneqq\Fht$ and $T\coloneqq\tT$.
By \cite[Theorems~1 and 2]{lenoir1986optimal}, the maps $\Fhk$ and $\Fhk^{-1}$ satisfy the following estimates
\begin{subequations}
	\label{Prop-FT}
	\begin{align}
		& |\Fhk|_{W^{m,\infty}(\hat{T})} \lesssim h_T^m  \text{ and } |\Fhk^{-1}|_{W^{m,\infty}(T)} \lesssim h_T^{-m},\quad 1 \le m\le k+1, \label{Prop-FT:eq1}\\
		&J_{\Fhk}(\hx) \approx  h_T^3 \text{ for all } \hx\in\hT,  \label{Prop-FT:eq2}
	\end{align}
\end{subequations}
where $h_T:=\operatorname{diam}(\Ftk^{-1}(T))$ and $\Ftk\coloneqq \Fhk\circ \Fht^{-1}:\tT\to T$. Since $\Fht$ is the linear nodal interpolant of $\Fhk$, the interpolation estimate \cite[Theorem 4.4.4]{brenner2008the} and \eqref{Prop-FT:eq1} give $\|\Fhk-\Fht\|_{W^{2,\infty}(\hT)}\lesssim h_T^2$. Scaling and the chain rule then yield
\begin{subequations}
\label{eq:Prop-Ftk}
\begin{align}
&|\Ftk-\operatorname{id}_{\tT}|_{W^{m,\infty}(\tT)}+|\Ftk^{-1}-\operatorname{id}_{T}|_{W^{m,\infty}(T)}\lesssim h_T^{2-m},\quad 0 \leq m \leq 2,\label{eq:Prop-Ftk-eq1}\\
&|J_{\Ftk}(\tx)-1|\lesssim h_T~\text{for all}~\tx\in \tT.\label{eq:Prop-Ftk-eq2}
\end{align}
\end{subequations}
Let $\calTh:= \{\Ftk(\tT):\tT\in\tcalTh\}$ and $\Omega_h\coloneqq{\rm int}\Big(\cup_{ T\in  \calTh} \overline{ T}\Big)$ be the isoparametric tetrahedral mesh and computational domain, respectively. By compatibility results in \cite[Lemmas~2 and 3]{lenoir1986optimal}, define $\Ftkh:\tOmega_h\to\Omega_h$ satisfying $\Ftkh|_{\tT}=\Ftk$. 

For each $T\in\calTh$, define a map $\Fkr:T\to\mathbb{R}^3$ which maps $T$ diffeomorphically onto a curvilinear tetrahedron $\rT$ exactly fitting $\Omega$. If $T$ has a boundary edge or face, $\Fkr$ is defined by \cite[Equation~(32)]{lenoir1986optimal}; otherwise, $\Fkr$ is the identity map on $T$. It follows from \cite[Propositions~2 and 3]{lenoir1986optimal} that
\begin{subequations}
\label{eq:psi-estimates}
\begin{align}
&|\Fkr-\operatorname{id}_T|_{W^{m,\infty}(T)}
+|\Fkr^{-1}-\operatorname{id}_{\rT}|_{W^{m,\infty}(\rT)}
\lesssim h_T^{k+1-m},\quad0\leq m\leq k+1, \label{eq:psi-estimates-1}\\
& | J_{\Fkr}(x)-1|\lesssim  h_T^k~\text{for all}~x\in T. \label{eq:psi-estimates-2}
\end{align}
\end{subequations}
The curvilinear tetrahedral mesh $\check{\calTh}:=\{\Fkr(T): T\in\calTh\}$ forms an exact partition of $\Omega$, and the compatibility of the local maps $\Fkr$ implies that they can be pieced together into a global mapping $\Theta_h:\Omega_h\to\Omega$.
Define $\Fhr\coloneqq\Fkr\circ \Fhk: \hT \to \rT$. The chain rule, together with \eqref{eq:psi-estimates} and \eqref{Prop-FT}, yields
\begin{subequations}
	\label{eq:theta-estimates}
	\begin{align}
		&|\Fhr|_{W^{m,\infty}(\hT)}\lesssim h_T^m \text{ and } |\Fhr^{-1}|_{W^{m,\infty}(\rT)} \lesssim h_T^{-m},\quad 1\le m\le k+1, \label{eq:theta-estimates-1}\\
		& J_{\Fhr}(\hx) \approx  h_T^3 \text{ for all } \hx\in\hT.\label{eq:theta-estimates-2}
	\end{align}
\end{subequations}

\subsection{Alfeld split}
To ensure the inf-sup stability of the proposed divergence-free method, the Alfeld split~\cite{schenck2014splines,guzman2018inf,fu2020exact} is introduced on each element of $\calTh$.
Let $\hT^{r}:= \{\hat K_i \}_{i=1}^4$ denote the Alfeld split of $\hT$, obtained by connecting its four vertices with its barycenter. For $\tT\in\tcalTh$ and $T\in\calTh$, define the corresponding local refinements~by
$$
\tT^r:=\{\Fht(\hK):\,\hK\in\hT^r\},\qquad T^r:=\{\Fhk(\hK):\,\hK\in\hT^r\}.
$$
Define the refined isoparametric mesh $\calTh^r:=\{K:K\in T^r, T\in\calTh\}$. The finite element spaces introduced in subsequent sections are defined on $\calTh$ rather than on $\calTh^r$; see \cite[Remark 2.2]{neilan2021divergence} for more details.

\subsection{Surface mappings}\label{subsec:facemapping}
Unlike the two-dimensional curved triangulations in \cite{neilan2021divergence}, where the common edge of two adjacent macroelements is always straight, the common face of two adjacent macroelements in three dimensions may be curved, as illustrated in Figure~\ref{fig:curved-interface-2D-3D}. To address the resulting analytical challenges, the surface mappings and associated geometric quantities are introduced below.

Let $\calFh$ and $\calFhI$ denote the sets of all and interior faces of $\calTh$, respectively. For $F\in\calFhI$ shared by $T^\pm\in\calTh$, define the face patch $\omega_F:= \{ T^+, T^-\}$ and let $\bn_F$ be the unit normal field pointing from $T^+$ to $T^-$. For a piecewise smooth function $\bv$, set $\bv^\pm:=\bv|_{T^\pm}$ and define its jump and average across $F$ by $\jump{\bv} := \bv^+|_F-\bv^-|_F$ and $\aver{\bv} := \frac12(\bv^+|_F+\bv^-|_F)$.
Similarly, define $\aver{\partial_{\bn_F}\bv}:=\frac12 \left(\partial_{\bn_F}\bv^+|_F +\partial_{\bn_F}\bv^-|_F\right)$ with $\partial_{\bn_F}\bv^\pm:=\nabla\bv^\pm\bn_F$. 
For $F\in\calFh\setminus\calFhI$, let $\bn_F$ be the outward unit normal field to $\Omega_h$. Then the jump $\jump{\bv}$ and average $\aver{\bv}$ reduce to $\bv|_F$, and $\aver{\partial_{\bn_F}\bv}:=\partial_{\bn_F}\bv|_F$. Define $h_F:=\operatorname{diam}(\Ftkh^{-1}(F))$.

Let $\bI$ be the identity matrix of $\mathbb{R}^{3\times3}$, and write $a\otimes b:=a b^\trans$ for $a,b\in\mathbb{R}^3$. Given a face $F$, define the matrix-valued function $\bP_F:=\bI-\bn_F\otimes\bn_F$. The tangential gradient of a smooth scalar function $g$ on $F$ is defined by $\nabla_F g:=\bP_F(\nabla g^e)|_F$, where $g^e$ is any smooth extension of $g$ to a neighborhood of $F$.
This definition is independent of the choice of the extension $g^e$; see \cite[Lemma~2.4]{dziuk2013finite}. For a tangential field $\btau$ on $F$, set $\partial_{\btau}g:=\btau\cdot\nabla_{F}g$. For a smooth vector-valued function $\bg=(g_1,g_2,g_3)^\trans$ on $F$, define $\nabla_F\bg:=(\nabla_F g_1,\nabla_F g_2,\nabla_F g_3)^\trans$.

Let the face mapping $\phi: F_0\to F_1$ be induced by an element mapping $\Phi:T_0\to T_1$ with $F_i$ a face of $T_i$, $i=0,1$, $\Phi(F_0)=F_1$, and $\phi=\Phi|_{F_0}$. In this paper, the lowercase letters of the corresponding element mapping denote the face mappings. Define its tangential Jacobian $\Dtau\phi:=\nabla_{F_0}\phi$, which satisfies $\Dtau\phi=(D\Phi|_{F_0})\bP_{F_0}$. Define its surface Jacobian $\mu_\phi$ by
\begin{equation}\label{eq:def-surface-mu}
\int_{F_0}(g\circ\phi)\,\mu_{\phi}\,\ds=\int_{F_1}g\,\ds \qquad\text{for all}~g\in L^1(F_1).    
\end{equation}

\begin{figure}[htbp]
\centering

\begin{minipage}[b]{0.42\textwidth}
\centering
\begin{tikzpicture}[scale=2.8, line join=round, line cap=round]
  
\coordinate (A) at (0,0);
\coordinate (B) at (1,0);

\coordinate (C) at ({2/3},{1/2});
\coordinate (D) at ({2/3},{-1/2});

\coordinate (Zu) at (0.55,0.18);
\coordinate (Zd) at (0.55,-0.18);

\draw[thick] (A) -- (C);
\draw[ultra thick] (C) .. controls (0.82,0.40) and (0.95,0.18) .. (B);

\draw[thick] (A) -- (D);
\draw[ultra thick] (D) .. controls (0.82,-0.40) and (0.95,-0.18) .. (B);

\draw[thick] (A) -- (B);

\draw[densely dashed] (Zu) -- (A);
\draw[densely dashed] (Zu) -- (B);
\draw[densely dashed] (Zu) -- (C);

\draw[densely dashed] (Zd) -- (A);
\draw[densely dashed] (Zd) -- (B);
\draw[densely dashed] (Zd) -- (D);

\fill (Zu) circle (0.5pt);
\fill (Zd) circle (0.5pt);

\end{tikzpicture}
\end{minipage}
\hfill
\begin{minipage}[b]{0.52\textwidth}
\centering

\tdplotsetmaincoords{76}{34}

\begin{tikzpicture}[scale=2.0, tdplot_main_coords, line join=round, line cap=round]

\coordinate (A) at (0,0,0);
\coordinate (B) at (1,0.5,0);
\coordinate (C) at (1.0,-0.5,0);

\coordinate (P) at (0,0,2/3);
\coordinate (Q) at (0,0,-2/3);

\coordinate (Zp) at (0.50,0,0.18);
\coordinate (Zq) at (0.46,-0.08,-0.18);

\coordinate (Mbc) at (1.135,0,-0.015);

\fill[gray!18]
  (A) -- (B)
  .. controls (1.09,0.36,0.03) and (1.135,0.18,0.005) .. (Mbc)
  -- cycle;

\fill[gray!18]
  (A) -- (Mbc)
  .. controls (1.135,-0.18,-0.035) and (1.09,-0.36,-0.05) .. (C)
  -- cycle;

\draw[thick] (A) -- (B);
\draw[thick] (A) -- (C);
\draw[ultra thick] (B)
  .. controls (1.18,0.22,0.06) and (1.18,-0.22,-0.15) .. (C);

\draw[thick] (P) -- (A);
\draw[ultra thick] (P) .. controls (0.30,0.48,0.78) and (0.78,0.56,0.22) .. (B);
\draw[ultra thick] (P) .. controls (0.30,-0.48,0.78) and (0.78,-0.56,0.22) .. (C);

\draw[thick] (Q) -- (A);
\draw[ultra thick] (Q) .. controls (0.30,0.48,-0.78) and (0.78,0.56,-0.22) .. (B);
\draw[ultra thick]
  (Q) .. controls (0.14,-0.20,-0.66) and (0.76,-0.54,-0.22) .. (C);

\draw[densely dashed] (Zp) -- (A);
\draw[densely dashed] (Zp) -- (B);
\draw[densely dashed] (Zp) -- (C);
\draw[densely dashed] (Zp) -- (P);

\draw[densely dashed] (Zq) -- (A);
\draw[densely dashed] (Zq) -- (B);
\draw[densely dashed] (Zq) -- (C);
\draw[densely dashed] (Zq) -- (Q);

\fill (Zp) circle (0.5pt);
\fill (Zq) circle (0.5pt);

\end{tikzpicture}
\end{minipage}

\caption{
Adjacent curved macroelements in two and three dimensions.
Dashed lines indicate the macroelement subdivisions.
}
\label{fig:curved-interface-2D-3D}
\end{figure}
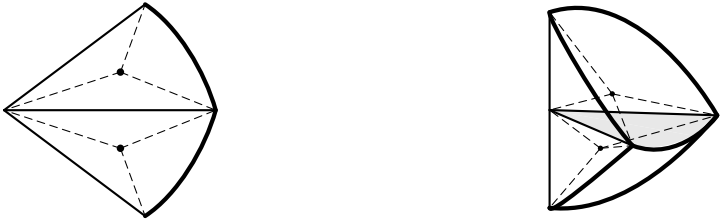

\subsection{Function transformations}
\label{subsec:transformations}
For a regular diffeomorphism $\Phi:~T_0 \rightarrow T_1$,  let $\mathcal{P}_\Phi$ denote the push-forward and $\mathcal{R}_\Phi:=\mathcal{P}_\Phi^{-1}$ the corresponding pullback. Superscripts will be used to distinguish the different transformations considered below.

Define the matrix-valued function $A_\Phi:= D\Phi  / J_\Phi$. Given vector functions $\bv_0$ on $T_0$ and $\bv_1$ on $T_1$, define the Piola transform and its inverse by
\begin{equation}
	\label{eq:Piola-V}
	\mathcal{P}_\Phi^V \bv_0 	:= (A_\Phi \bv_0)\circ \Phi^{-1}, \quad 
	\mathcal{R}_\Phi^V \bv_1	:= A_\Phi^{-1}(\bv_1 \circ \Phi).
\end{equation}
For scalar functions $q_0$ on $T_0$ and $q_1$ on $T_1$, define the composition transformations
\begin{equation}
	\label{eq:Piola-Q}
	\mathcal{P}_\Phi^Q q_0 := q_0 \circ \Phi^{-1},
	\qquad
	\mathcal{R}_\Phi^Q q_1	:= q_1 \circ \Phi .
\end{equation}
For another diffeomorphism $\Psi:T_1\to T_2$, these transformations are compatible with composition; namely
\begin{equation}\label{eq:pushforwcomposition}
\mathcal{P}_{\Psi\circ \Phi}^{X} = \mathcal{P}_{\Psi}^{X}\mathcal{P}_{\Phi}^{X}
\quad\text{and}\quad
\mathcal{R}_{\Psi\circ \Phi}^{X} = \mathcal{R}_{\Phi}^{X} \mathcal{R}_{\Psi}^{X}\quad\text{ for }X\in\{V,Q\}.
\end{equation}
For an element $T$, define $\bH(\Div;T):=\{\bv\in\bL^2(T):\,\Div\bv\in L^2(T)\}$.
\begin{lemma}\label{lem:pro-AT}
For $\bv_0\in \bH(\Div;T_0)$ and $q_0\in L^2(T_0)$, set $\bv_1:=\mathcal P_\Phi^V\bv_0$ and $q_1:=\mathcal P_\Phi^Qq_0$. Then $\bv_1\in \bH(\Div;T_1)$, $q_1\in L^2(T_1)$ and 
\begin{equation}
\label{eq:pro-AT}
\Div\bv_1\circ\Phi=\Div\bv_0/J_{\Phi},\qquad
(q_1, \Div \bv_1)_{T_1} = (q_0, \Div  \bv_0)_{T_0}.
\end{equation}
For a face $F_0$ of $T_0$, let $\phi=\Phi|_{F_0}$. For smooth function $\bv_0$ on $T_0$, the following normal-trace identity holds
\begin{equation}\label{eq:face-equation-AT}
\bv_0 \cdot \bn_0 = \mu_{\phi}(\bv_1\cdot\bn_1)\circ\phi\quad\text{ on }~F_0,
\end{equation}
where $\bn_i$ denotes the outward unit normal field to $T_i$, $i=0,1$.
\end{lemma}
\begin{proof}
It follows from \cite[Lemma 2.1.6]{boffi2013mixed} that \eqref{eq:pro-AT} holds.
For a face $F_0$ of $T_0$, let $F_1:=\phi(F_0)\subset\partial T_1$. The normal-flux preservation of the Piola transform gives
\begin{equation*}
\int_{F_1}\bv_1 \cdot \bn_1 \, p_1 \ds=\int_{F_0}\bv_0 \cdot \bn_0\, p_0\ds \quad\text{ for all }p_0\in L^2(F_0),~~ p_1=p_0\circ\phi^{-1}.
\end{equation*}
Applying \eqref{eq:def-surface-mu} with $g=\bv_1 \cdot \bn_1 \, p_1$ to the left-hand side of the above equation gives
$$
\int_{F_0}\big(\mu_{\phi}(\bv_1 \cdot \bn_1)\circ\phi\big) \, p_0 \ds=\int_{F_1}\bv_1 \cdot \bn_1 \, p_1 \ds=\int_{F_0}\bv_0 \cdot \bn_0p_0\ds.
$$
This and the arbitrariness of $p_0$ prove \eqref{eq:face-equation-AT}.
\end{proof}

The following results give bounds on $A_\Phi$ and its inverse for $\Phi=\Fhk$, which follows from similar procedures as in \cite[Lemma 2.3]{neilan2021divergence} together with \eqref{Prop-FT}.
\begin{lemma}\label{lem:ATBounds}
Given $T\in\calTh$ and $\Fhk:\hT\to T$, it holds that
\begin{align*}
	|A_{\Fhk}|_{W^{m,\infty}(\hT)}\lesssim  h_T^{m-2},~m\geq0,\quad
	|A_{\Fhk}^{-1}|_{W^{m,\infty}(\hT)}
	\lesssim \left\{
	\begin{array}{ll}
		h_T^{2+m}, & 0\leq m\leq 2k-2,\\
		0, & 2k-1 \le m.
		\end{array}
	\right.
\end{align*}
\end{lemma}
Additionally, the following scaling results from \cite[Lemma 2.3]{bernardi1989optimal} will be used.

\begin{lemma}
\label{lem::NormInThT}
Let $T\in\calTh$ and $\bv\in\bW^{m,p}(K)$ for any $K\in T^r$, with $0\le m\le k+1$ and $1\le p\le+\infty$. Set $\hbv=\bv\circ\Phi_T$ and $\hK=\Fhk^{-1}(K)$, then it holds that
\begin{align*}
|\bv|_{W^{m,p}(K)}\lesssim  h_T^{3/p-m}  \|\hbv\|_{W^{m,p}(\hK)},
\quad
|\hbv|_{W^{m,p}(\hK)} \lesssim  h_T^{m-3/p} \|\bv\|_{W^{m,p}(K)}.
\end{align*}
\end{lemma}

The two lemmas above give the following result.
\begin{lemma}
\label{lem::normH1XThXT}
Let $T\in\calTh$ and $\bv\in \bW^{m,p}(K)$ for any $K\in T^r$, with $0\le m\le k+1$ and $1\le p\le+\infty$. Set $\hK=\Fhk^{-1}(K)$, then it holds that
\begin{align*}
|\bv|_{W^{m,p}(K)}\lesssim  h_T^{3/p-m-2} \|\Rhkv\bv\|_{W^{m,p}(\hK)},
\quad
|\Rhkv\bv|_{W^{m,p}(\hK)} \lesssim  h_T^{m+2-3/p} \|\bv \|_{W^{m,p}(K)}.
\end{align*}
\end{lemma}

\section{The Curved Scott--Vogelius pair}
\label{sec:global-spaces}
For polynomial degree $k\geq3$, this section constructs a Scott--Vogelius pair on three-dimensional curved meshes following the framework of  \cite{neilan2021divergence,durst2024a}, and establishes its main properties.

\subsection{Local spaces} Recall that $\hT^r=\{\hK_1,\hK_2,\hK_3,\hK_4\}$ is the Alfeld split of the reference tetrahedron $\hT$. Define
\begin{equation*}
\begin{aligned}
\hbV (\hT) &:= \{\hbv\in \bH^1(\hT):\ \hbv|_{\hK}\in \bP_{k}(\hK) \text{ for all } \hK\in \hT^{r}\}, \\
\hQ (\hT) &:= \{\hq \in L^2(\hT):\ \hq|_{\hK} \in P_{k-1}(\hK) \text{ for all } \hK\in \hT^{r}\}.    
\end{aligned}
\end{equation*}
Using the Piola and composition transformations defined in \eqref{eq:Piola-V} and \eqref{eq:Piola-Q}, respectively, define the local spaces on each affine tetrahedron $\tT\in\tcalTh$ by
$$
\tbV(\tT):= \Phtv \hbV (\hT), \qquad   \tQ(\tT) := \Phtq   \hQ (\hT).
$$
Since $\Fht$ is affine, $\tbV(\tT)$ and $\tQ(\tT)$ are piecewise polynomial on $\tT^r$.
For each tetrahedron $T\in\calTh$, possibly with curved boundary, define
$$
\begin{aligned}
&\bV(T) := \Phkv \hbV (\hT), \qquad   &&Q(T) := \Phkq   \hQ (\hT).
\end{aligned}
$$
If $\Fhk$ is affine, then $\bV(T)=\tbV(\tT)$ and $Q(T)=\tQ(\tT)$; otherwise, both $\bV(T)$ and $Q(T)$ are not necessarily piecewise polynomial spaces.
The fact that $\Fhk=\Ftk\circ\Fht$ and the composition property \eqref{eq:pushforwcomposition} give the following relations
\begin{subequations}\label{eq:push-pull-tv-V}
\begin{align}
\bV(T)=\Ptkv\tbV(\tT),\qquad\quad Q(T)=\Ptkq\tQ(\tT),\label{eq:push-tV-V}\\
\tbV(\tT)=\Rtkv\bV(T),\qquad\quad \tQ(\tT)=\Rtkq Q(T).\label{eq:pull-V-tV}
\end{align}    
\end{subequations}
Define the subspaces of $\hbV(\hT)$ and $\hQ(\hT)$ by
$$
\begin{aligned}
&\hbV_0(\hT):=\hbV(\hT)\cap \bH^1_0(\hT),\qquad &\hQ_0(\hT):=\hQ(\hT)\cap L^2_0(\hT).
\end{aligned}
$$
The corresponding subspaces of $\bV(T)$ and $Q(T)$ are defined by
$$
\begin{aligned}
\bV_0(T) := \Phkv \hbV_0 (\hT), \qquad   Q_0(T) := \Phkq   \hQ_0 (\hT).
\end{aligned}
$$
Note that $\bV_0(T)\subset\bH^1_0(T)$, whereas $Q_0(T)$ may not be a subspace of $L^2_0(T)$. 
In~fact,
$$
\int_T {q}/{(J_{\Fhk}\circ\Fhk^{-1})}\dx = \int_{\hT}  \Rhkq q \hdx=0 \quad\text{ for all } q\in Q_0(T).
$$

The following local inf-sup stability follows from the argument of \cite[Theorem~3.9]{neilan2021divergence}, which will be used later.

\begin{lemma}\label{lem:local-stability}
Given $q\in Q_0(T)$, there exists $\bv\in\bV_0(T)$ such that
$$ \Div\bv = {h_T^3q}/{(J_{\Fhk}\circ \Fhk^{-1})}\quad
\text{ and } \quad
\|\bv\|_{H^1(T)} \lesssim \|q\|_{L^2(T)}.
$$
\end{lemma}
\begin{proof}
For $q\in Q_0(T)$, let $\hq:=\Rhkq q\in\hQ_0(\hT)$. By
\cite[Theorem~3.3]{guzman2018inf}, there exists
$\hbv\in\hbV_0(\hT)$ such that
$$
\Div\hbv=h_T^3\hq,
\qquad
\|\hbv\|_{H^1(\hT)}
\lesssim h_T^3\|\hq\|_{L^2(\hT)}.
$$
Set $\bv:=\Phkv\hbv\in\bV_0(T)$. The Piola identity in \eqref{eq:pro-AT} gives
$$
\Div\bv
=({\Div\hbv}/{J_{\Fhk}})\circ\Fhk^{-1}
={h_T^3q}/{(J_{\Fhk}\circ\Fhk^{-1})}.
$$
Since $\hbv=\Rhkv\bv$ and $\hq=q\circ \Fhk$, the scaling estimates in Lemmas~\ref{lem::normH1XThXT} and~\ref{lem::NormInThT} yield
$$
\|\bv\|_{H^1(T)}
\lesssim h_T^{-3/2}\|\hbv\|_{H^1(\hT)}
\lesssim h_T^{3/2}\|\hq\|_{L^2(\hT)}
\lesssim \|q\|_{L^2(T)},
$$
which completes the proof.
\end{proof}

\subsection{Degrees of freedom}\label{sec:DOFs-VT}
Recall that a function in $\hbV(\hT)$ is uniquely determined by its values at the following Lagrange nodes:
(\romannumeral 1). the five vertices of $\hT^r$;
(\romannumeral 2). the $k-1$ points on each of the ten open edges of $\hT^r$;
(\romannumeral 3). the $(k-1)(k-2)/2$ points on each of the ten open triangular faces of~$\hT^r$;
(\romannumeral 4). the $(k-1)(k-2)(k-3)/6$ points in each of the four open subtetrahedra~of~$\hT^r$.
In (\romannumeral 3) and (\romannumeral 4), the nodes must be selected so that the corresponding nodal values uniquely determine functions in $P_{k-3}$ on each face and in $P_{k-4}$ on each subtetrahedron, respectively.
Denote this nodal set by $\hNT :=\{\hat{a}_i\}_{i=1}^{n_k}$, where $n_k$ is the total number of these nodes.

For each $\tT\in\tcalTh$ and the corresponding element $T\in\calTh$, define the node sets
\begin{equation*}
\tNT:=\Fht(\hNT),\qquad\NT:=\Fhk(\hNT).   
\end{equation*}
Then $\NT=\Ftkh(\tNT)$. The unisolvence of the reference nodal degrees of freedom (DOFs) for $\hbV(\hT)$ and the invertibility of $A_{\Fhk}$ imply the following lemma. 
\begin{lemma}\label{lm:dofs-VT}
$\bv\in\bV(T)$ is uniquely determined by its values $\bv(a)$ for all $a\in\NT$.
\end{lemma}
Since the nodal DOFs are unisolvent on both $\bV(T)$ and $\tbV(\tT)$, the nodal correspondence defines an isomorphism $\ST:\tbV(\tT)\to\bV(T)$ by
$$
(\ST\tbv)(a)=\tbv(\ta)\quad\text{for all }\ta\in\tNT,\ a=\Ftkh(\ta).  
$$
The arguments in \cite[Lemma 3.5]{neilan2021divergence} and \cite[Lemma 4.4]{durst2024a}, together with Lemmas~\ref{lem:ATBounds} and \ref{lem::NormInThT}, yield the following interpolation estimate.
\begin{lemma}\label{eq:interpolation-error}
Let $\bIT:\bH^2(T)\to\bV(T)$ be the interpolation operator determined by $(\bIT\bu)(a)=\bu(a)$ for all $a\in\NT$. Then for $\bu\in\bH^s(T)$ with $s\geq2$, it holds that
$$
\|\bu-\bIT\bu\|_{H^m(K)}\lesssim h_T^{l-m}\|\bu\|_{H^{l}(T)}\quad\text{for all~} K\in T^r,\ 0\leq m\leq l:=\min\{k+1,s\}.
$$
\end{lemma}

\subsection{Global spaces}
Recall the Scott--Vogelius pair on the affine triangulation $\tcalTh$ as
$$
\begin{aligned}
\tbV_h &= \{ \tbv_h \in \bH_0^1(\tOmega_h):\ \tbv_h |_{\tT} \in \tbV(\tT) \text{ for all } \tT \in \tcalTh \}, \\
\tQ_h & = \{ \tq_h \in L_0^2(\tOmega_h):\ \tq_h |_{\tT} \in \tQ(\tT) \text{ for all } \tT \in \tcalTh  \}.
\end{aligned}
$$
The Stokes pair $\tbV_h\times\tQ_h$ is inf-sup stable and divergence-free \cite{guzman2018inf}. 
Define the operators $\Sh$ and $\Gh$ elementwise by
$$\Sh|_T=\ST:\tbV(\tT)\to\bV(T)\quad\text{and}\quad\Gh|_T=\Ptkq:\tQ(\tT)\to Q(T),$$
for all $T=\Ftk(\tT)\in\calTh$.
The global spaces on $\Omega_h$ are then given by
$$
\begin{aligned}
\bV_h:=\{\bv_h:\ \bv_h=\Sh\tbv_h,\,\tbv_h\in\tbV_h\},\quad Q_h:=\{q_h:\ q_h=\Gh\tq_h,\,\tq_h\in\tQ_h\}.
\end{aligned}
$$
\begin{remark}\label{re:Vh-Con-Sh-ISO}
In fact, $\bV_h$ consists of functions that belong to $\bV(T)$ on each $T\in\calTh$, are continuous at the nodes specified in Lemma~\ref{lm:dofs-VT}, and vanish on $\partial\Omega_h$.
Besides, $\Sh:\tbV_h\to\bV_h$ is an isomorphism, since each local map $\ST$ is an isomorphism.
\end{remark}

The Piola transform and nodal continuity ensure that $\bV_h$ is $\bH(\Div)$-conforming, even across curved interior faces (see Figure~\ref{fig:curved-interface-2D-3D}). The proof below follows from that of \cite[Lemma~3.5]{demlow2025taylor} for a Taylor--Hood method for the surface Stokes problem and is provided here for completeness.
\begin{theorem}
\label{thm:Hdiv-conformity}
The global velocity space $\bV_h$ satisfies
$$
\bV_h \subset \bH_0(\Div;\Omega_h) := \{ \bv_h \in \bL^2(\Omega_h): \Div \bv_h \in L^2(\Omega_h), \bv_h \cdot \bn|_{\partial \Omega_h} =0\}.
$$
\end{theorem}
\begin{proof}
Set $\bv_h\in\bV_h$. Let $F=\partial T^+\cap\partial T^-\in\calFhI$. Set $\tT^\pm:=\Ftkh^{-1}(T^\pm)$,
$\bv^\pm:=\bv_h|_{T^\pm}$, and
$\tbv^\pm:=\mathcal R_{\Psi_{T^\pm}}^V\bv^\pm
\in\tbV(\tT^\pm)$,
where the inclusion follows from \eqref{eq:pull-V-tV}.
Thus $\tbv^\pm|_{\tF}\in \bP_k(\tF)$ for
$\tF:=\partial\tT^+\cap\partial\tT^-$.
Set $\psi_F:=\Ftkh|_{\tF}:\tF\to F$.
Let $\bn^\pm$ and $\tilde{\bn}^\pm$ be the outward unit normals
to $T^\pm$ and $\tT^\pm$.
By the normal-trace identity~\eqref{eq:face-equation-AT},
$$
\tbv^+\cdot\tilde{\bn}^+ +\tbv^-\cdot\tilde{\bn}^-
= \mu_{\psi_F} (\bv^+\cdot\bn^++\bv^-\cdot\bn^-)\circ\psi_F
\quad\text{on }\tF.
$$
The left-hand side belongs to $P_k(\tF)$ and vanishes at all
Lagrange nodes of degree~$k$ on $\tF$, since $\bv_h$ is continuous
at their images under $\psi_F$ and $\bn^+=-\bn^-$ on $F$.
Unisolvence and $\mu_{\psi_F}>0$ therefore imply
$\bv^+\cdot\bn^++\bv^-\cdot\bn^-=0$ on $F$.
Since $\bv_h|_T\in\bH^1(T)$ for each $T\in\calTh$ and
$\bv_h=0$ on $\partial\Omega_h$, it follows that
$\bv_h\in\bH_0(\Div;\Omega_h)$.
\end{proof}

The velocity space $\bV_h$ is a subspace of
$\bH_0(\Div;\Omega_h)$ but not of $\bH_0^1(\Omega_h)$.
This results in a consistency error that may reduce the convergence rates unless functions in $\bV_h$ have sufficient weak continuity across interior faces.
In two dimensions, Durst and Neilan~\cite{durst2024a} obtain the required weak continuity by choosing Gauss--Lobatto points as the nodal DOFs on interior edges.
Since such points do not exist on triangular faces~\cite{xu2011on}, their analysis cannot be carried over directly to three dimensions.
To recover the optimal convergence rates, this paper adds consistency terms on each interior face but without any penalty term, which is used to ensure the coercivity of the resulting bilinear form.
This is guaranteed by the enhanced jump estimate established in the next subsection.

\subsection{Enhanced jump estimate}
Define the isoparametric Lagrange finite element space of degree $k$ by
$
\bV_h^{\mathrm{iso}}:=\{\bv_h:\bv_h=\tbv_h\circ\Ftkh^{-1},\,\tbv_h\in\tbV_h\},
$
which is a subspace of $\bH^1_0(\Omega_h)$.
The isomorphism $\Sh$ induces the operator $\Eh:\bV_h\to\bV_h^{\mathrm{iso}}\subset\bH^1_0(\Omega_h)$,
$$
\Eh\bv_h:=(\Sh^{-1}\bv_h)\circ\Ftkh^{-1}\quad\text{ for }\bv_h\in\bV_h.
$$
The following lemma extends \cite[Lemma 4.5]{neilan2021divergence} to three dimensions, which applies to the enhanced jump estimate below.

\begin{lemma}
\label{lemma:enrichment-estimate}
For sufficiently small $h$,
every $\bv_h\in\bV_h$ and $T\in\calTh$ satisfy
\begin{equation*}
\sum_{m=0,1}
h_T^m\|\bv_h-\Eh\bv_h\|_{H^m(T)}\lesssim h_T^2\min\{\|\nabla\bv_h\|_{L^2(T)},\|\nabla(\Eh\bv_h)\|_{L^2(T)}\}.
\end{equation*}
\end{lemma}
\begin{proof}
For each $T\in\calTh$, the definitions of $\Sh$ and $\Eh$ yield
\begin{equation}\label{eq:node-equive-vh-Ehvh}
\Eh\bv_h|_T(a)=\bv_h|_T(a)\qquad\text{for all }a\in\NT.    
\end{equation}
If $T$ is affine, both functions belong to the
piecewise polynomial space of degree $k$, and hence coincide by unisolvence.
Thus the estimate holds in this case. 
 
Suppose that $T$ is curved. Set $\hbv_T:=\Rhkv(\bv_h|_T)$ and $\hbw_T:=(\Eh\bv_h|_T)\circ\Fhk$,  then $\hbv_T,\hbw_T\in\hbV(\hT)$. This and \eqref{eq:node-equive-vh-Ehvh} give
$$
A_{\Fhk}(\ha)\hbv_T(\ha)=\hbw_T(\ha)
\qquad\text{for all }\ha\in\hNT.
$$
Thus $\hbw_T$ is the piecewise degree-$k$ Lagrange interpolant of $A_{\Fhk}\hbv_T$ on $\hT^{r}$. The Bramble--Hilbert lemma therefore yields
$$
\|A_{\Fhk}\hbv_T-\hbw_T\|_{H^m(\hK)} \lesssim |A_{\Fhk}\hbv_T|_{H^{k+1}(\hK)},\quad \hK\in\hT^{r},~m=0,1.$$
Since $\hbv_T|_{\hK}$ is polynomial of degree at most $k$, the right-hand side is bounded by the estimates of $A_{\Fhk}$ in Lemma~\ref{lem:ATBounds} and finite-dimensional norm equivalence
$$
\begin{aligned}
|A_{\Fhk}\hbv_T|_{H^{k+1}(\hK)} &\lesssim
\sum_{j=0}^k
|A_{\Fhk}|_{W^{k+1-j,\infty}(\hK)}
|\hbv_T|_{H^j(\hK)}  \\&\lesssim h_T^{-1}\|\hbv_T\|_{H^{k}(\hK)}\lesssim
h_T^{-1}\|\hbv_T\|_{L^2(\hK)}\lesssim
h_T^{-1/2}\|\bv_h\|_{L^2(T)},
\end{aligned}
$$
where the last inequality follows from Lemma~\ref{lem::normH1XThXT}.
The three-dimensional scaling estimates in Lemma \ref{lem::NormInThT} then give for $m=0,1$
\begin{equation}\label{eq:vhEh-con-vh-L2}
\begin{aligned}
\|\bv_h-\Eh\bv_h\|_{H^m(T)}
\lesssim h_T^{3/2-m}\|A_{\Fhk}\hbv_T-\hbw_T\|_{H^m(\hT)}
\lesssim h_T^{1-m}\|\bv_h\|_{L^2(T)}. 
\end{aligned}
\end{equation}
By the mesh construction, $T$ has a boundary vertex
$a=\Fhk(\hat a)$. Since $\bv_h(a)=0$ and $A_{\Fhk}(\hat a)$ is
invertible, $\hbv_T(\hat a)=0$. Finite-dimensional norm
equivalence yields
$\|\hbv_T\|_{L^2(\hT)}
\lesssim |\hbv_T|_{H^1(\hT)}.$
Together with Lemma~\ref{lem::normH1XThXT}, this gives
\begin{equation*}
\begin{aligned}
\|\bv_h\|_{L^2(T)}&\lesssim h_T^{-1/2}\|\hbv_T\|_{L^2(\hT)}\lesssim h_T^{-1/2}|\hbv_T|_{H^1(\hT)}\lesssim h_T(\|\bv_h\|_{L^2(T)}+\|\nabla\bv_h\|_{L^2(T)}).
\end{aligned}
\end{equation*}
For sufficiently small $h$, the above inequality yields $\|\bv_h\|_{L^2(T)}\lesssim h_T\|\nabla\bv_h\|_{L^2(T)}$.
This and \eqref{eq:vhEh-con-vh-L2} imply $h_T^m\|\bv_h-\Eh\bv_h\|_{H^m(T)}\lesssim h_T^{2}\|\nabla\bv_h\|_{L^2(T)}$. Taking $m=1$ and applying the triangle inequality imply $\|\nabla\bv_h\|_{L^2(T)} \lesssim\|\nabla(\Eh\bv_h)\|_{L^2(T)}$, which completes the proof.
\end{proof}

\begin{lemma}[Enhanced jump estimate]
\label{corollary:jump-estimate}
For sufficiently small $h$, it holds that
\begin{equation}
\label{estimate:local-jump}
h_F^{-1/2}
\|\jump{\bv_h}\|_{L^2(F)}
\lesssim
h_F
\sum_{T\in\omega_F}
\|\nabla\bv_h\|_{L^2(T)}\quad\text{for all }\bv_h\in\bV_h,~ F\in \calFhI.
\end{equation}
\end{lemma}

\begin{proof}
Since $\Eh\bv_h\in\bH_0^1(\Omega_h)$, $\jump{\bv_h} = \jump{\bv_h-\Eh\bv_h}$ holds for each $F\in\calFhIc$.
The trace inequality and Lemma~\ref{lemma:enrichment-estimate} yield
$$
\begin{aligned}
h_F^{-1/2} \|\jump{\bv_h}\|_{L^2(F)}\lesssim \sum_{T\in\omega_F}\sum_{m=0,1}
h_T^{m-1}\|\bv_h-\Eh\bv_h\|_{H^m(T)}
\lesssim h_F \sum_{T\in\omega_F} \|\nabla\bv_h\|_{L^2(T)}, 
\end{aligned}
$$
which completes the proof.
\end{proof}

Define the broken $H^1$-seminorm
$$
\|\bv_h\|_{1,h}^2
:=
\sum_{T\in\calTh}
\|\nabla\bv_h\|_{L^2(T)}^2.
$$
Lemma~\ref{lemma:enrichment-estimate} and the Poincar\'e inequality for $\Eh\bv_h\in\bH_0^1(\Omega_h)$ give
\begin{equation}\label{eq:discrete-Poincare}
\|\bv_h\|_{L^2(\Omega_h)} \lesssim
\|\bv_h-\bE_h\bv_h\|_{L^2(\Omega_h)} + \|\bE_h\bv_h\|_{L^2(\Omega_h)}
\lesssim \|\bv_h\|_{1,h},    
\end{equation}
for all $\bv_h \in \bV_h$.
This implies that $\|\cdot\|_{1,h}$ is a norm in $\bV_h$.

\subsection{Global inf-sup stability}
This subsection establishes the global inf-sup condition by combining the local stability result in Lemma~\ref{lem:local-stability} with Stenberg's macroelement technique. Throughout this subsection, $h$ is assumed sufficiently small.

Define the spaces of piecewise constants with respect to $\tcalTh$ and $\calTh$ by
$$
\begin{aligned}
\tY_h:&=\{\tq_h\in L^2_0(\tOmega_h):\, \tq_h|_{\tT}\in P_0(\tT)\,\text{ for all }\tT\in\tcalTh\}\subset\tQ_h,\\
Y_h:&=\{q_h:\, q_h=\Gh\tq_h,~\tq_h\in\tY_h\}\subset Q_h,
\end{aligned}
$$
Suitable face-bubble functions on the curved interior faces in the following lemma are used to prove the stability of $\bV_h\times Y_h$.

\begin{lemma}\label{lem:face-bubbles}
For $F\in\calFhI$, there is a face bubble function $\bb_F\in\bV_h$, supported on the union of two elements in $\omega_F$ and vanishing on every face in $\calFh\setminus\{F\}$, such that
\begin{equation}\label{eq:face-bubble-properties}
\int_F\bb_F\cdot\bn_F\ds=1
\quad
\text{and}
\quad
\|\bb_F\|_{1,h}\lesssim h_F^{-3/2}.
\end{equation}
\end{lemma}
\begin{proof}
Given $F\in\calFhI$, let $T\in\omega_F$ such that $\bn_F$ is outward to $T$. Set $\hat F:=\Fhk^{-1}(F)\subset\partial\hT$ and $\phi_F:=\Fhk|_{\hat F}:\hF\to F$.
Let $\hat\lambda_1,\hat\lambda_2,\hat\lambda_3$ be the barycentric coordinates on $\hat F$, and define
$\hat b:=\hat\lambda_1\hat\lambda_2\hat\lambda_3/(\int_{\hat F}\hat\lambda_1\hat\lambda_2\hat\lambda_3\hds).$ Define $\bb_F\in\bV_h$ by
$$
\bb_F(a)=
\hat b(\hat a)\bn_F(a)/\mu_{\phi_F}(\hat a),\quad
a\in F\cap\NT,\quad\hat a=\phi_F^{-1}(a),
$$
and $\bb_F(a)=0$ for $a\in\big(\bigcup_{T\in\calTh}\NT\big)\setminus F$.
Thus $\bb_F$ is supported on the union of two elements in $\omega_F$, and vanishes on every face in $\calFh\setminus\{F\}$.
Set $\hat{\bb}_{\hF}:=\Rhkv(\bb_F|_T)\in\hbV(\hT)$.
By the normal-trace identity \eqref{eq:face-equation-AT}, $\hat{\bb}_{\hF}|_{\hF}\cdot\hat{\bn}_{\hF}\in P_k(\hat F)$ has the same nodal values as $\hat b\in P_3(\hat F)\subset P_k(\hat F)$.
Unisolvence and the normal-flux preservation of the Piola transform yield
$$\int_F\bb_F\cdot\bn_F\ds=\int_{\hat F}\hat{\bb}_{\hF}\cdot\hat{\bn}_{\hF}\,\hds=\int_{\hat F}{\hat b}\,\hds=1.$$
Since $\mu_{\phi_F}\approx h_F^2$ and
$\|A_{\Fhk}^{-1}\|_{L^\infty(\hT)}\lesssim h_T^2$, the nodal values of $\hat{\bb}_{\hF}$ are bounded. It follows from the finite-dimensional norm equivalence and Lemma~\ref{lem::normH1XThXT} that $|\bb_F|_{H^1(T)}\lesssim h_T^{-3/2}\|\hat\bb_{\hF}\|_{H^1(\hT)}\lesssim h_T^{-3/2}\sum_{\hat{a}\in\hNT}|\hat{\bb}_{\hF}(\hat a)|\lesssim h_F^{-3/2}$. The same scaling arguments apply to the other element in $\omega_F$ and complete the proof of \eqref{eq:face-bubble-properties}.
\end{proof}

\begin{lemma}
\label{lem:stability-Vh-Yh}
Let $k\geq3$. For any $q_h\in Y_h$, there exists $\bv_h\in\bV_h$, such that
\begin{equation}
\label{eq:Yh-stability}
(\Div\bv_h,q_h)_{\Omega_h}\gtrsim \|q_h\|_{L^2(\Omega_h)}^2
\quad\text{ and }\quad
\|\bv_h\|_{1,h}\lesssim\|q_h\|_{L^2(\Omega_h)}.
\end{equation}
\end{lemma}
\begin{proof}
Let $q_h=\Gh\tq_h\in Y_h$ with $\tq_h\in\tY_h$.
By \eqref{eq:Prop-Ftk}, $\|q_h\|_{L^2(\Omega_h)}\approx\|\tq_h\|_{L^2(\tOmega_h)}.$
By \cite{girault1979finite}, there exists $\tbw\in\bH^1_0(\tOmega_h)$ such that $\tDiv\tbw=\tq_h$ and $\|\tbw\|_{H^1(\tOmega_h)}\lesssim \|\tq_h\|_{L^2(\tOmega_h)}$.
The hidden constant is independent of $h$ by \cite[Theorem~4.4]{bernardi2016contiuity} and the geometric estimates for the mapping $\Theta_h\circ\Ftkh:\tOmega_h\to\Omega$.
Since $k\geq3$, $\tbV_h$ contains the Bernardi--Raugel space \cite{bernardi1985analysis}, which forms an inf-sup stable pair with the piecewise constant pressure space.
The proof of this stability result shows that there exists $\tbv_h\in\tbV_h$ such that the integral of the normal component of $\tbv_h$ over each face of $\tcalTh$ equals that of $\tbw$, and the $H^1$-norm of $\tbv_h$ is uniformly controlled by that of $\tbw$.
As $\tq_h$ is piecewise constant, the divergence theorem gives
\begin{subequations}\label{eq:affine-Yh-stability}
\begin{align}
(\tDiv\tbv_h,\tq_h)_{\tOmega_h}&=(\tDiv\tbw,\tq_h)_{\tOmega_h}=\|\tq_h\|_{L^2(\tOmega_h)}^2\gtrsim\|q_h\|_{L^2(\Omega_h)}^2,\label{eq:affine-Yh-stability-a}\\
\|\tbv_h\|_{H^1(\tOmega_h)}
&\lesssim\|\tbw\|_{H^1(\tOmega_h)}\lesssim\|\tq_h\|_{L^2(\tOmega_h)}\lesssim\|q_h\|_{L^2(\Omega_h)}.\label{eq:affine-Yh-stability-b} 
\end{align}
\end{subequations}
For each $F\in\calFhI$, set $\tF:=\Ftkh^{-1}(F)$ and choose $\tbn_{\tF}$ with orientation compatible with $\bn_F$. Let $\bb_F$ be the face bubble function in Lemma~\ref{lem:face-bubbles}. Define $\bv_h\in\bV_h$ as
$$
\bv_h:=\Sh\tbv_h+\sum_{F\in\calFhI}d_F\bb_F \quad\text{with}\quad
d_F:=\int_{\tF}\tbv_h\cdot\tbn_{\tF}\,\tds-\int_F\Sh\tbv_h\cdot\bn_F\,\ds.
$$
By \eqref{eq:face-bubble-properties} and $\bb_F$ vanishing on all the other faces,
$\int_F\bv_h\cdot\bn_F\ds=\int_{\tF}\tbv_h\cdot\tbn_{\tF}\tds$ holds for all $F\in\calFh$.
Since $q_h$ and $\tq_h$ are constant with the same value on corresponding elements, integration by parts on each element and \eqref{eq:affine-Yh-stability-a} give
\begin{equation}\label{eq:divvh-control-qh-0}
(\Div\bv_h,q_h)_{\Omega_h}
=(\tDiv\tbv_h,\tq_h)_{\tOmega_h}=\|\tq_h\|_{L^2(\tOmega_h)}^2
\gtrsim\|q_h\|_{L^2(\Omega_h)}^2,
\end{equation}
which proves the first inequality of \eqref{eq:Yh-stability}.
Set $\bw_h:=\Eh(\Sh\tbv_h)=\tbv_h\circ\Ftkh^{-1}$ and $\boldsymbol e_h:=\bw_h-\Sh\tbv_h$. The estimates for $\Ftk$ in \eqref{eq:Prop-Ftk} give $\|\bw_h\|_{H^1(T)}\approx\|\tbv_h\|_{H^1(\tT)}$. Together with Lemma~\ref{lemma:enrichment-estimate} and the triangle inequality, this yields
\begin{align}
\|\boldsymbol e_h\|_{L^2(T)}
+h_T\|\boldsymbol e_h\|_{H^1(T)}
&\lesssim h_T^2\|\nabla\tbv_h\|_{L^2(\tT)},
\label{eq:bubbleScal0}\\
\|\Sh\tbv_h\|_{1,h}
&\lesssim\|\tbv_h\|_{H^1(\tOmega_h)}.
\label{eq:bubbleScal1}
\end{align}
Fix $F\in\calFhI$ and $T\in\omega_F$.
Since $\bw_h\circ\Ftk=\tbv_h$ and
$\boldsymbol e_h=\bw_h-\Sh\tbv_h$, a change of variables gives
$d_F=\int_{\tF}\tbv_h\cdot\bigl(\bI-J_{\Ftk}(D\Ftk)^{-\trans}\bigr) \tbn_{\tF}\,\tds+\int_F\boldsymbol e_h\cdot\bn_F\,\ds.$
The estimates in \eqref{eq:Prop-Ftk} imply
$\|\bI-J_{\Ftk}D\Ftk^{-\trans}\|_{L^\infty(\tT)}\lesssim h_T$.
Using the Cauchy--Schwarz inequality, the trace inequality, the inverse estimate, and
\eqref{eq:bubbleScal0} give
$$
\begin{aligned}
|d_F|&\lesssim h_F^2\|\tbv_h\|_{L^2(\tF)}+h_F\|\boldsymbol e_h\|_{L^2(F)}\\
&\lesssim 
h_F^{3/2}\|\tbv_h\|_{L^2(\tT)}+h_F^{1/2}
\bigl(\|\boldsymbol e_h\|_{L^2(T)}+h_T\|\nabla\boldsymbol e_h\|_{L^2(T)}
\bigr)\lesssim h_F^{3/2}\|\tbv_h\|_{L^2(\tT)}.
\end{aligned}
$$
Combining \eqref{eq:face-bubble-properties}, the finite overlap of the
face patches, \eqref{eq:bubbleScal1}, and \eqref{eq:affine-Yh-stability-b} gives
$$
\begin{aligned}
\|\bv_h\|_{1,h}
\lesssim\|\Sh\tbv_h\|_{1,h}
+\Bigl(\sum_{F\in\calFhI}h_F^{-3}|d_F|^2\Bigr)^{1/2}
\lesssim\|\tbv_h\|_{H^1(\tOmega_h)}
\lesssim\|q_h\|_{L^2(\Omega_h)}.
\end{aligned}
$$
Together with \eqref{eq:divvh-control-qh-0}, this completes the proof.
\end{proof}
The local stability in Lemma~\ref{lem:local-stability} and the stability of $\bV_h\times Y_h$ in Lemma~\ref{lem:stability-Vh-Yh}
yield the following global inf-sup condition. 
Similar arguments can be found in \cite{neilan2021divergence}.

\begin{theorem}
\label{theorem:global-infsup}
There exists a constant $\beta>0$, independent of
$h$, such that
\begin{equation}
\label{eq:global-infsup}
\inf_{q_h\in Q_h\setminus\{0\}}
\sup_{\bv_h\in\bV_h\setminus\{\bzero\}}
\frac{(q_h,\Div\bv_h)_{\Omega_h}}
{\|q_h\|_{L^2(\Omega_h)}\|\bv_h\|_{1,h}}
\geq\beta.
\end{equation}
\end{theorem}

\begin{proof}
Let $q_h\in Q_h$, and define an elementwise constant
$\bar q_h$ by $\int_T (q_h-\bar q_h)/(J_{\Fhk}\circ\Fhk^{-1})\,\dx=0$ for all $T\in\calTh$.
Then $(q_h-\bar q_h)|_T\in Q_0(T)$.
For each $T\in\calTh$, Lemma~\ref{lem:local-stability}
provides $\bv_T\in\bV_0(T)$ such that
$\Div\bv_T = ({h_T^3(q_h-\bar q_h)})/ ({J_{\Fhk}\circ\Fhk^{-1}})$ and 
$\|\bv_T\|_{H^1(T)} \lesssim\|q_h-\bar q_h\|_{L^2(T)}$.
Set $\bv_0|_T:=\bv_T$. Since $\bv_T=0$ on $\partial T$,
it follows that $\bv_0\in\bV_h$ and
$\|\bv_0\|_{1,h} \lesssim\|q_h-\bar q_h\|_{L^2(\Omega_h)}$.
The divergence theorem and the fact that $\bar q_h$ is
elementwise constant give
$(\Div\bv_0,\bar q_h)_{\Omega_h}=0$.
Therefore, by \eqref{Prop-FT:eq2},
$$
\begin{aligned}
(\Div\bv_0,q_h)_{\Omega_h} =(\Div\bv_0,q_h-\bar q_h)_{\Omega_h} =\sum_{T\in\calTh}
\int_T \frac{h_T^3|q_h-\bar q_h|^2} {J_{\Fhk}\circ\Fhk^{-1}}\,\dx
\gtrsim \|q_h-\bar q_h\|_{L^2(\Omega_h)}^2.
\end{aligned}
$$
Consequently,
\begin{equation}\label{eq:inf-sup1}
\|q_h-\bar q_h\|_{L^2(\Omega_h)}
\lesssim
\sup_{\bv_h\in\bV_h\setminus\{\bzero\}}
\frac{(\Div\bv_h,q_h)_{\Omega_h}}
{\|\bv_h\|_{1,h}}.
\end{equation}
Since $\Gh^{-1}\bar q_h$ consists of the affine element
averages of $\Gh^{-1}q_h$ and $\Gh^{-1}q_h$ has zero mean on
$\tOmega_h$, $\bar q_h\in Y_h$. 
Lemma~\ref{lem:stability-Vh-Yh} and  the
Cauchy--Schwarz inequality yield
$$
\begin{aligned}
\|\bar q_h\|_{L^2(\Omega_h)}
&\lesssim
\sup_{\bv_h\in\bV_h\setminus\{\bzero\}}
\frac{(\Div\bv_h,\bar q_h)_{\Omega_h}}
{\|\bv_h\|_{1,h}}
 \lesssim
\sup_{\bv_h\in\bV_h\setminus\{\bzero\}}
\frac{(\Div\bv_h,q_h)_{\Omega_h}}
{\|\bv_h\|_{1,h}}
+\|q_h-\bar q_h\|_{L^2(\Omega_h)}.\\
\end{aligned}
$$
Combining the two estimates gives
$$
\|q_h\|_{L^2(\Omega_h)}
\leq
\|q_h-\bar q_h\|_{L^2(\Omega_h)}
+\|\bar q_h\|_{L^2(\Omega_h)}
\lesssim
\sup_{\bv_h\in\bV_h\setminus\{\bzero\}}
\frac{(\Div\bv_h,q_h)_{\Omega_h}}
{\|\bv_h\|_{1,h}},
$$
which proves \eqref{eq:global-infsup}.
\end{proof}

\section{Discrete Stokes Method}
\label{sec:discrete-method}
By introducing consistency terms to the bilinear form, this section gives the discrete scheme and establishes its well-posedness and optimal-order error estimates.

\subsection{Discrete scheme}
Assume that the boundary $\partial\Omega$ and the source $\bft$ for the Stokes equation \eqref{eq:continuous-mixed} are sufficiently smooth, such that the solution $(\bu,p)\in \bH^s(\Omega)\times H^{s-1}(\Omega)$ with $s\geq2$,
and can be extended to $\mathbb{R}^3$ \cite{kato2000extension} in a way such that
$(\bu,p)\in \bH^s(\mathbb{R}^3)\times H^{s-1}(\mathbb{R}^3)$ with $\Div \bu=0$,
\begin{equation}
\label{extenUandP}
\|\bu\|_{H^s(\mathbb{R}^3)}\lesssim  \|\bu\|_{H^s(\Omega)},\qquad \|p\|_{H^{s-1}(\mathbb{R}^{3})}\lesssim  \|p\|_{H^{s-1}(\Omega)}.    
\end{equation}
Then, set ${\bft} := -\nu \Delta \bu+ \nabla p\in \bH^{s-2}(\mathbb{R}^3)$, and let ${\bft}_h\in \bL^2(\Omega_h)$ be a computable approximation of ${\bft}|_{\Omega}$.
Let $\bV(h):=\bV_h + \bH^2(\Omega_h)$. For any $\bu_h,\bv_h\in\bV(h)$,~define
$$
a_h(\bu_h,\bv_h)\coloneqq\sum_{T\in\calTh}(\nabla\bu_h,\nabla\bv_h)_T - \sum_{F \in \calFhIc} \int_F \aver{\partial_{\bn_F}\bu_h}  \cdot \jump{\bv_h} +  \aver{\partial_{\bn_F}\bv_h}  \cdot \jump{\bu_h}\rd s.
$$
The discrete scheme for \eqref{eq:continuous-mixed} seeks $(\bu_h,p_h)\in\bV_h\times Q_h$ such that
\begin{subequations}
\label{eq:discrete-stokes}
\begin{align}
\nu a_h(\bu_h,\bv_h) -(p_h, \Div \bv_h)_{\Omega_h}
&=(\bft_h,\bv_h)_{\Omega_h}
&&\text{ for all }\bv_h\in\bV_h,\label{eq:ah-uh-vh}\\
(\Div \bu_h,q_h)_{\Omega_h}&=0
&&\text{ for all } q_h\in Q_h.\label{eq:divuh-qh}
\end{align}
\end{subequations}
For any $\bv\in\bV(h)$, define the semi-norm
\begin{equation*}
\label{eq:augmented-seminorm}
\|\bv\|_{*,h}^2 :=\|\bv\|_{1,h}^2
+\sum_{F\in\calFhI} \left( h_F\|\aver{\partial_{\bn_F}\bv}\|_{L^2(F)}^2 +h_F^{-1}\|\jump{\bv}\|_{L^2(F)}^2
\right).
\end{equation*}
The weighted Cauchy--Schwarz inequality gives
\begin{equation}
\label{estimate:ah-bounded}
|a_h(\bu,\bv)| \lesssim  \|\bu\|_{*,h}\|\bv\|_{*,h}
 \qquad\text{ for all } \bu,\bv\in\bV(h).
\end{equation}
By the inverse trace inequality and the scaling arguments,
\begin{equation}\label{eq:average-bound-H1}
\sum_{F\in\calFhI} h_F \left\|\aver{\partial_{\bn_F}\bv_h}
\right\|_{L^2(F)}^2 \lesssim \|\bv_h\|_{1,h}^2
\qquad \text{ for all }\bv_h\in\bV_h.    
\end{equation}
This and the enhanced jump estimate in Lemma~\ref{corollary:jump-estimate} imply that 
\begin{equation}
\label{eq:equivalence}
\| \bv_h\|_{1,h} \approx \| \bv_h \|_{*,h} \qquad \text{ for all }  \bv_h \in \bV_h.
\end{equation}
\begin{theorem}[Well-posedness]
\label{theorem:well-posedness}
For sufficiently small $h$, the discrete problem \eqref{eq:discrete-stokes} has a
unique solution $(\bu_h,p_h)\in\bV_h\times Q_h$. 
\end{theorem}
\begin{proof}
Set $s_h(\bu_h,\bv_h) := \sum_{F\in\calFhIc} \int_F \aver{\partial_{\bn_F}\bu_h}\cdot \jump{\bv_h}  \rd s$.
The weighted Cauchy--Schwarz inequality, \eqref{eq:average-bound-H1}, and Lemma~\ref{corollary:jump-estimate} give $|s_h(\bu_h,\bv_h)|\lesssim  h \|\bu_h\|_{1,h} \|\bv_h\|_{1,h}.$
Therefore, for sufficiently small $h$, the following coercivity holds
\begin{equation}
\label{estimate:ah-coercive}
a_h(\bv_h,\bv_h)=\|\bv_h\|_{1,h}^2 - 2s_h(\bv_h,\bv_h)\gtrsim\|\bv_h\|_{1,h}^2
\quad\text{ for all }\bv_h\in\bV_h.
\end{equation}
Moreover, \eqref{estimate:ah-bounded} and \eqref{eq:equivalence} imply the continuity $|a_h(\bu_h,\bv_h)| \lesssim  \|\bu_h\|_{1,h}\|\bv_h\|_{1,h}$ for all $\bu_h,\bv_h \in\bV_h.$
According to Babu\v{s}ka--Brezzi theory \cite{boffi2013mixed}, this, \eqref{estimate:ah-coercive}, and Theorem~\ref{theorem:global-infsup} complete the proof.
\end{proof}

The discrete divergence constraint implies the divergence-free property. This is a simple extension of \cite[Lemma~5.2]{neilan2021divergence} to curved tetrahedral meshes.

\begin{lemma}
\label{cor:exact-divergence-free}
Let $\bu_h\in\bV_h$ satisfy \eqref{eq:divuh-qh}.
Then $\Div\bu_h=0$ in $\Omega_h$.
\end{lemma}

\subsection{Error estimates}
\label{subsection:error-estimates}
Define $\bZ_h = \left\{ \bv_h\in\bV_h: \Div\bv_h=0 \text{ in }\Omega_h\right\}.$
By Lemma \ref{cor:exact-divergence-free} and the discrete problem, the discrete velocity solution $\bu_h \in \bV_h$ can be uniquely determined by the following problem: Find $\bu_h \in \bZ_h$ such that 
\begin{equation}
\label{eq:discrete-Zh}
\nu a_h(\bu_h, \bv_h) = (\bft_h, \bv_h)_{\Omega_h}  \quad \text{ for all } \bv_h \in \bZ_h.    
\end{equation}

\begin{lemma}
\label{lemma:boundary-compatible-approximation}
Let $\bu\in\bH^s(\Omega)\cap\bV$ with $s\geq2$, and let
$\bu$ also denote an extension satisfying \eqref{extenUandP}.
For sufficiently small $h$, there exists
$\bw_h\in\bV_h$ such that
\begin{equation}\label{estimate:boundary-compatible-approximation}
\|\bu-\bw_h\|_{*,h}
\lesssim h^{l-1}\|\bu\|_{H^l(\Omega)},
\qquad l=\min\{k+1,s\}.
\end{equation}
\end{lemma}
\begin{proof}
Set $\bw_h|_T:=\bIT(\bu\circ\Fkrh)$ for every $T\in\calTh$. Since $\bu\circ\Fkrh=0$ on $\partial\Omega_h$, boundary unisolvence then implies $\bw_h\in\bV_h$. The scaled trace inequalities yield
$$
\|\bu-\bw_h\|_{*,h}^2
\lesssim
\sum_{T\in\calTh}\sum_{m=0}^2 h_T^{2(m-1)}
\sum_{K\in T^r}\|\bu-\bw_h\|_{H^m(K)}^2.
$$
On each $T\in\calTh$, inserting $\bIT\bu$ gives, for
$K\in T^r$ and $m=0,1,2$,
$$
\|\bu-\bw_h\|_{H^m(K)}
\leq
\|\bu-\bIT\bu\|_{H^m(K)}
+\|\bIT(\bu-\bu\circ\Fkrh)\|_{H^m(K)}.
$$
The first term is the local interpolation error, which
can be estimated by Lemma~\ref{eq:interpolation-error}.
For the second term, scaling arguments and
finite-dimensional norm equivalence on the reference
element yield
$$
\|\bIT(\bu-\bu\circ\Fkrh)\|_{H^m(K)}
\lesssim h_T^{3/2-m}
\max_{a\in\NT}|\bu(a)-\bu(\Fkrh(a))|.
$$
By \eqref{eq:psi-estimates-1},
$|a-\Fkrh(a)|\lesssim h_T^{k+1}$.
Sobolev embedding and, for $s\geq3$, the mean value
formula give
$$
\max_{a\in\NT}|\bu(a)-\bu(\Fkrh(a))|
\lesssim
\begin{cases}
h_T^{(k+1)/2}\|\bu\|_{H^2(\mathbb R^3)},
& s=2,\\
h_T^{k+1}\|\bu\|_{H^3(\mathbb R^3)},
& s\geq3.
\end{cases}
$$
Since $\sum_{T\in\calTh}h_T^3\approx|\Omega|$,
$k\geq3$, and $l\leq k+1$, combining the above estimates complete the proof.
\end{proof}

\begin{theorem}
\label{theorem:error-estimate}
Let $(\bu,p)\in\bH^s(\Omega)\times H^{s-1}(\Omega)$, $s\geq2$,
solve \eqref{eq:continuous-mixed}, and let
$(\bu_h,p_h)\in\bV_h\times Q_h$ solve \eqref{eq:discrete-stokes}.
For sufficiently small $h$ and $l:=\min\{k+1,s\}$,
\begin{equation}
\label{estimate:velocity-error}
\|\bu-\bu_h\|_{1,h}
\lesssim h^{l-1}\|\bu\|_{H^l(\Omega)}
+\nu^{-1}|\bft-\bft_h|_{\bZ_h^*},
\end{equation}
where $|\bft-\bft_h|_{\bZ_h^*}:=
\sup_{\bv_h\in\bZ_h\setminus\{\bzero\}}
(\bft-\bft_h,\bv_h)_{\Omega_h}/ \|\bv_h\|_{1,h}$.
The pressure approximation satisfies
\begin{equation}
\label{estimate:pressure-error}
\begin{aligned}
\|p-p_h\|_{L^2(\Omega_h)}
\lesssim{}&
\nu\bigl(\|\bu-\bu_h\|_{1,h}
+h^{l-1}\|\bu\|_{H^l(\Omega)}\bigr)\\
&+\inf_{q_h\in Q_h}\|p-q_h\|_{L^2(\Omega_h)}
+\|\bft-\bft_h\|_{L^2(\Omega_h)}.
\end{aligned}
\end{equation}
\end{theorem}

\begin{proof}
For $\bz_h\in\bZ_h$, the coercivity \eqref{estimate:ah-coercive}, continuity \eqref{estimate:ah-bounded} and equivalence \eqref{eq:equivalence} give 
$$\|\bz_h-\bu_h\|_{1,h}^2\lesssim \|\bz_h-\bu\|_{*,h}\|\bz_h-\bu_h\|_{1,h}+|a_h(\bu-\bu_h,\bz_h-\bu_h)|.$$
Since $\bz_h-\bu_h\in\bZ_h$, the above inequality and the triangle inequality yield
\begin{equation}
\label{es1:eq1}
\nu\|\bu-\bu_h\|_{1,h}\lesssim
\nu\inf_{\boldsymbol z_h\in\bZ_h}\|\bu-\boldsymbol z_h\|_{*,h}
+\nu\sup_{\bv_h\in\bZ_h\setminus\{\bzero\}}
\frac{|a_h(\bu_h-\bu,\bv_h)|}{\|\bv_h\|_{1,h}}.   
\end{equation}
Let $\bw_h$ be the approximation of $\bu$ from
Lemma~\ref{lemma:boundary-compatible-approximation}.
By Theorem~\ref{theorem:global-infsup} and \cite[Corollary~4.1.1]{boffi2013mixed}, there exists $\br_h\in\bV_h$ such that
$$
(\Div\br_h,q_h)_{\Omega_h}
=(\Div(\bw_h-\bu),q_h)_{\Omega_h}~\text{ for all } q_h\in Q_h,
\quad
\|\boldsymbol r_h\|_{1,h}\lesssim\|\bu-\bw_h\|_{1,h}.
$$
Since $\Div\bu=0$, Lemma~\ref{cor:exact-divergence-free} implies $\bw_h-\boldsymbol r_h\in\bZ_h$. This, \eqref{eq:equivalence}, and \eqref{estimate:boundary-compatible-approximation} yield
\begin{equation}
\label{es1:eq2}
\nu\inf_{\bz_h\in\bZ_h}\|\bu-\bz_h\|_{*,h}
\leq\nu\|\bu-(\bw_h-\br_h)\|_{*,h}\lesssim\nu\|\bu-\bw_h\|_{*,h}
\lesssim\nu h^{l-1}\|\bu\|_{H^l(\Omega)}.
\end{equation}
For $\bv_h\in\bV_h$, recall that $\bv_h=0$ on $\partial\Omega_h$ and $\bv_h\in\bH_0(\Div;\Omega_h)$. Moreover, $\jump{\bu}=0$ for all $F\in\calFhI$.
Elementwise integration by parts gives
\begin{equation}\label{eq:consistency}
\nu a_h(\bu,\bv_h)-(p,\Div\bv_h)_{\Omega_h}
=(\bft,\bv_h)_{\Omega_h}
\quad\text{ for all } \bv_h\in\bV_h.
\end{equation}
Together with \eqref{eq:discrete-Zh}, this implies
$\nu a_h(\bu_h-\bu,\bv_h)=(\bft_h-\bft,\bv_h)_{\Omega_h}$
for all $\bv_h\in\bZ_h$.
Substituting this and \eqref{es1:eq2} into \eqref{es1:eq1}
proves \eqref{estimate:velocity-error}.

Let $q_h\in Q_h$. The triangle inequality and Theorem~\ref{theorem:global-infsup} give
$$
\|p-p_h\|_{L^2(\Omega_h)}
\lesssim\|p-q_h\|_{L^2(\Omega_h)}
+\sup_{\bv_h\in\bV_h\setminus\{\bzero\}}
\frac{|(p_h-q_h,\Div\bv_h)_{\Omega_h}|}{\|\bv_h\|_{1,h}}.
$$
Subtracting \eqref{eq:consistency} from the discrete problem \eqref{eq:ah-uh-vh} gives
$$
(p_h-q_h,\Div\bv_h)_{\Omega_h}=\nu a_h(\bu_h-\bu,\bv_h)+(p-q_h,\Div\bv_h)_{\Omega_h}-(\bft_h-\bft,\bv_h)_{\Omega_h}.
$$
The splitting $\bu_h-\bu=(\bu_h-\bw_h)+(\bw_h-\bu)$, together with \eqref{estimate:ah-bounded}, \eqref{eq:equivalence} and the triangle inequality, yields
$$
\begin{aligned}
|a_h(\bu_h-\bu,\bv_h)|
&\lesssim\bigl(\|\bu_h-\bw_h\|_{*,h}+\|\bu-\bw_h\|_{*,h}\bigr)\|\bv_h\|_{1,h}\\
&\lesssim\bigl(\|\bu_h-\bu\|_{1,h}+\|\bu-\bw_h\|_{*,h}\bigr)\|\bv_h\|_{1,h}.    
\end{aligned}
$$
The estimates \eqref{estimate:velocity-error}, \eqref{estimate:boundary-compatible-approximation},
the Cauchy--Schwarz inequality, the discrete Poincar\'e inequality
\eqref{eq:discrete-Poincare}, and the infimum over $q_h\in Q_h$
now prove \eqref{estimate:pressure-error}.
\end{proof}

Assume that $\bu$ and $\bft$ are sufficiently smooth. If $\bft_h$ is taken as its nodal (isoparametric) interpolant of degree $k$, the discrete Poincar\'e inequality \eqref{eq:discrete-Poincare} gives
$$| \bft- \bft_h|_{\bZ_h^*}\lesssim \|\bft-\bft_h\|_{L^2(\Omega_h)}\lesssim h^{k+1}\|\bft\|_{H^{k+1}(\Omega)}.$$
Substituting this into Theorem~\ref{theorem:error-estimate} leads to the following velocity error estimate:
\begin{equation}
\label{es1:velocityByIntf}
\begin{aligned}
\| \nabla(\bu-\bu_h) \|_{L^2(\Omega_h)}\lesssim& h^{k} \big(\| \bu \|_{H^{k+1}(\Omega)}+ \nu^{-1}h\|\bft\|_{H^{k+1}(\Omega)}\big).
\end{aligned}    
\end{equation}
The presence of the terms $\nu^{-1} h \|\bft\|_{H^{k+1}(\Omega)}$ in \eqref{es1:velocityByIntf} indicates that the velocity error remains coupled to the pressure scaling and inversely proportional to the viscosity. Consequently, a direct application of the pair $(\bV_h, Q_h)$ fails to provide a pressure-robust scheme, as large pressure gradients or small viscosity values may significantly degrade the accuracy of the velocity approximation.

\section{Pressure-Robust Modification}
\label{sec:pressure-robust}

This section develops a unified framework for constructing commuting operators on curved meshes. These operators are then used to define a computable approximation $\bft_h$ and to obtain a pressure-robust scheme.

\subsection{Abstract framework}
Introduce the spaces defined on $\hT$ as 
$$
\begin{aligned}
\hat\Sigma (\hT)
&:=\{\hat\sigma\in C^1(\hT):
\hat\sigma|_{\hK}\in P_{k+2}(\hK)
\text{ for all }\hK\in\hT^r\},\\
\hbW(\hT)
&:=\{\hbw\in \bC^0(\hT):
\hat{\operatorname{curl}}\hbw\in \bC^0(\hT),\ \hbw|_{\hK}\in\bm P_{k+1}(\hK)
\text{ for all }\hK\in\hT^r\}. 
\end{aligned}
$$
Given $\Phi: T_0 \to T_1$, for scalar functions $\sigma_0$ on $T_0$ and $\sigma_1$ on $T_1$, define 
\begin{equation}\label{eq:defPsandRs}
\mathcal{P}_\Phi^\Sigma \sigma_0 := \sigma_0 \circ \Phi^{-1}, \quad \mathcal{R}_\Phi^\Sigma \sigma_1 := \sigma_1 \circ \Phi.     
\end{equation}
For vector functions $\bw_0$ on $T_0$ and $\bw_1$ on $T_1$, define
\begin{equation}\label{eq:defPwandRw}
\mathcal{P}_\Phi^W \bw_0 := (D\Phi^{-\trans} \bw_0) \circ \Phi^{-1}, \quad 
\mathcal{R}_\Phi^W \bw_1 := D\Phi^{\trans} (\bw_1 \circ \Phi).    
\end{equation}
Define the local spaces by
$$
\Sigma(T):=\Phks \hat\Sigma(\hT), \quad \bW(T):=\Phkw \hbW(\hT)  \quad \text{ for all } T\in\calTh.
$$
The corresponding global spaces are defined as
$$
\begin{aligned}
\Sigma_h &:=\{\sigma_h\in H^1(\Omega_h):
\sigma_h|_T\in\Sigma(T)\text{ for all }T\in\calTh\},\\ 
\bW_h &:=\{\bw_h\in\bH(\operatorname{curl};\Omega_h):
\bw_h|_T\in\bW(T)\text{ for all }T\in\calTh\}.
\end{aligned}
$$
The following theorem establishes the commuting property and approximation estimate required for the pressure-robust load approximation. An analogous result was established in two dimensions by Neilan and Otus \cite[Theorem~6.1]{neilan2021divergence}.
\begin{theorem}\label{thm:abstract-pressure-robustness}
There exist linear operators $\pis:H^{4}(\Omega)\to\Sigma_h$ and $\piw:\bH^3(\Omega)\to\bW_h$ such that
\begin{subequations}\label{eq:abstract-properties}
\begin{align}
\piw\nabla p
&=\nabla\pis p
&&\text{for all }p\in H^{4}(\Omega),
\label{eq:abstract-commuting}\\
\|\bv-\piw\bv\|_{L^2(\Omega_h)}
&\lesssim h^k\|\bv\|_{H^{k}(\Omega)}
&&\text{for all }\bv\in\bH^{k}(\Omega),
\label{eq:abstract-approximation}
\end{align}
\end{subequations}
where $\bv$ in the left-hand side of the above inequality is an $\bH^{k}$ extension of $\bv|_\Omega$. 
\end{theorem}
As in \cite[Corollary~6.2]{neilan2021divergence}, these two properties immediately imply pressure robustness of the resulting scheme. The argument is included for completeness.
\begin{corollary}
\label{cor:pressure-robust}
Let $(\bu,p)\in\bH^{k+2}(\Omega)\times H^{k+1}(\Omega)$
be the solution of \eqref{eq:continuous-mixed} and let
$(\bu_h,p_h)$ be the solution of \eqref{eq:discrete-stokes} with $\bft_h = \piw \bft$. Then it holds that
\begin{equation}
\label{eq:pressure-robust-error}
\|\bu-\bu_h\|_{1,h}
\lesssim h^k\|\bu\|_{H^{k+2}(\Omega)}.
\end{equation}  
\end{corollary}
\begin{proof}
Recall that the extension of $\bft|_\Omega$ is given by $\bm f=-\nu\Delta\bu+\nabla p$. For every $\bv_h\in \bZ_h$, linearity and
\eqref{eq:abstract-commuting} give
$$
\begin{aligned}
(\bm f-\bm f_h,\bv_h)_{\Omega_h} &=-\nu(\Delta\bu-\piw\Delta\bu,\bv_h)_{\Omega_h}+(\nabla p-\piw\nabla p,\bv_h)_{\Omega_h}\\
&=-\nu(\Delta\bu-\piw\Delta\bu,\bv_h)_{\Omega_h}+(\nabla(p-\pis p),\bv_h)_{\Omega_h}\\ &=-\nu(\Delta\bu-\piw\Delta\bu,\bv_h)_{\Omega_h}.
\end{aligned}
$$
Here the last identity follows from $\bZ_h\subset\bH_0(\Div;\Omega_h)$ and $\Div\bv_h=0$; in particular, it uses the global $H^1$-continuity of $\pis p$.  Therefore, \eqref{eq:abstract-approximation} and \eqref{eq:discrete-Poincare} yield
$$
\begin{aligned}
|(\bm f-\bm f_h,\bv_h)_{\Omega_h}| \lesssim \nu h^k\|\Delta\bu\|_{H^{k}(\Omega)} \|\bv_h\|_{L^2(\Omega_h)}\lesssim \nu h^k\|\bu\|_{H^{k+2}(\Omega)}\|\bv_h\|_{1,h}.
\end{aligned}
$$
Consequently, $|\bm f-\bm f_h|_{\bZ_h^*} \lesssim \nu h^k\|\bu\|_{H^{k+2}(\Omega)}$. This and \eqref{estimate:velocity-error} prove \eqref{eq:pressure-robust-error}.
\end{proof}

\begin{remark}
The proof in Corollary \ref{cor:pressure-robust} shows that the pressure-robust error estimate above does not require $\piw\bv\in\bH(\operatorname{curl};\Omega_h)$ for general $\bv$. Nevertheless, the global $H(\operatorname{curl})$-conformity of $\piw$ is established in the following, thereby obtaining a conforming commuting pair on curved meshes.
\end{remark}

\subsection{Construction of the global operator}
\label{sec-construct-global-operator}

The operators are constructed elementwise by composing the pullback from the exact element to the reference tetrahedron, the commuting interpolation operators on the reference tetrahedron, and the corresponding pushforward to the curved computational element.

Define the local operators $\pist: H^4(\rT) \to \Sigma(T)$ and $\piwt: \bH^3(\rT) \to \bW(T)$ as 
\begin{equation}
\label{def:local-projectors}
\pist :=\Phks\hpis \Rhrs, \quad 
\piwt :=\Phkw\hpiw \Rhrw,
\end{equation}
where $\hpis: H^{4}(\hT)\to\hat\Sigma(\hT)$ and $\hpiw:\bH^3(\hT)\to\hbW(\hT)$ are the interpolation operators defined by the DOFs in \cite[Lemmas~4.8 and~4.11]{fu2020exact}. 

By \cite[Theorem~4.20]{fu2020exact}, the commuting property on the reference element holds
$$
\hpiw\hnab\hat\sigma= \hnab\hpis\hat\sigma \quad \text{ for all } \hat\sigma\in H^4(\hT).
$$
This, combined with $\Rhrw\nabla=\hnab\Rhrs$ and $\Phkw\hnab=\nabla\Phks$ by the chain rule, gives 
\begin{equation}
\label{eq:local-commuting-checked}
\begin{aligned}
\piwt\nabla\sigma =\Phkw\hpiw\hnab\Rhrs\sigma =\Phkw\hnab\hpis\Rhrs\sigma =\nabla\Phks\hpis\Rhrs\sigma =\nabla\pist\sigma,
\end{aligned}    
\end{equation}
for all $\sigma \in H^4(\rT)$, as illustrated in Figure~\ref{fig:local-commutativity}. 

\begin{figure}[htbp]
\centering
\scriptsize
\begin{tikzcd}[
  row sep=5em,
  column sep=5em
]
H^4(\rT)
  \arrow[r, "\Rhrs"]
  \arrow[d, "\nabla"']
&
H^4(\hT)
  \arrow[r, "\hpis"]
  \arrow[d, "\hnab"']
&
\hat\Sigma(\hT)
  \arrow[r, "\Phks"]
  \arrow[d, "\hnab"']
&
\Sigma(T)
  \arrow[d, "\nabla"]
\\
\bH^3(\rT)
  \arrow[r, "\Rhrw"']
&
\bH^3(\hT)
  \arrow[r, "\hpiw"']
&
\hbW(\hT)
  \arrow[r, "\Phkw"']
&
\bW(T)
\end{tikzcd}

\caption{Local commuting properties.}
\label{fig:local-commutativity}
\end{figure}

Based on \eqref{def:local-projectors}, define the global operators as
$$
(\pis\sigma)|_T:=\pist(\sigma|_{\rT}), \quad (\piw\bw)|_T:=\piwt(\bw|_{\rT})\quad  \text{ for all } T \in \calTh.  
$$
It remains to verify that the elementwise operators have the asserted
global conformity: for $\sigma\in H^4(\Omega)$ and $\bw\in\bH^3(\Omega)$,
\begin{equation}\label{eq:global-compatibility}
\pis\sigma\in H^1(\Omega_h),\qquad\piw\bw\in\bH(\operatorname{curl};\Omega_h).
\end{equation}

\subsection{The conformity of the global operators} 

Let $X\in\{\Sigma,W\}$. This subsection presents an abstract framework for verifying the conformity of $\pix$, followed by its verification for the scalar and vector interpolation operators, respectively.

\subsubsection{The abstract conformity of~\texorpdfstring{$\pix$}{Pi X h}}
\label{sec:verify-abstract-X}
For a face $F$ of an element $T$, recall that $\bP_F=\bI-\bn_F\otimes\bn_F$. Let $\TrxF$ denote the trace operator associated with $X$,
\begin{equation}\label{eq:defTrsandw}
\TrsF\sigma:=\sigma|_F,\qquad\TrwF\bw:=\bP_{F}(\bw|_F).    
\end{equation}
Since $\pix$ is defined elementwise in the corresponding local finite element space, the global conformity \eqref{eq:global-compatibility} reduces to the compatibility of the local interpolation operators across each interior face $F=\partial T^+\cap\partial T^-$ with $T^\pm\in\calTh$,
\begin{equation}\label{eq:conformity}
\TrxF (\pixtp v)=\TrxF (\pixtm v).
\end{equation}
Here, $v$ is defined on $\Omega$, with $v\in H^4(\Omega)$ for $X=\Sigma$ and $v\in\bH^3(\Omega)$ for $X=W$.

To verify \eqref{eq:conformity}, introduce the face mappings and transformations associated with element mappings and transformations. For a diffeomorphism $\Phi:T_0\to T_1$, let $\mathcal{P}^{X}_{\Phi}$ and $\mathcal{R}^{X}_{\Phi}$ be the transformations defined in \eqref{eq:defPsandRs} with $X=\Sigma$ and \eqref{eq:defPwandRw} with $X=W$. For a face $F_0$ of $T_0$, set $F_1:=\Phi(F_0)\subset\partial T_1$ and $\phi:=\Phi|_{F_0}:F_0\to F_1$. For a trace function $g_i$ on $F_i$, let $g_i^e$ be any smooth extension of $g_i$ to $T_i$ satisfying
\begin{equation}\label{eq:def_gi_ext}
g_i=\TrxFi g_i^e,\qquad i=0,1.
\end{equation}
Define the face pushforward and pullback by
\begin{equation}\label{eq:def-push-face-x}
\mathcal P_{\phi}^Xg_0 = \TrxFo\mathcal P_{\Phi}^Xg_0^e,\qquad
\mathcal R_{\phi}^Xg_1 = \TrxFz\mathcal R_{\Phi}^Xg_1^e.    
\end{equation}
The following lemma shows that the face transformations are well-defined and depend only on the face diffeomorphism $\phi$.

\begin{lemma}
For $X\in\{\Sigma,W\}$, the face transformations
$\mathcal P_\phi^X$ and $\mathcal R_\phi^X$ defined in
\eqref{eq:def-push-face-x} are well defined and depend
only on $\phi$, independently of the extensions $g_i^e$
and the element diffeomorphism $\Phi$ inducing $\phi$.
More precisely, for smooth scalar functions $g_i$ on $F_i$,
\begin{equation}\label{eq:def-face-s-pull}
\mathcal R_\phi^\Sigma g_1=g_1\circ\phi,
\qquad
\mathcal P_\phi^\Sigma g_0=g_0\circ\phi^{-1},
\end{equation}
and, for smooth tangential vector functions $g_i$ on $F_i$,
\begin{equation}\label{eq:def-face-w-pull}
\mathcal R_\phi^W g_1 = (\Dtau\phi)^\trans(g_1\circ\phi), \qquad \mathcal P_\phi^W g_0 = (\Dtau\phi^{-1})^\trans(g_0\circ\phi^{-1}).
\end{equation}
\end{lemma}
\begin{proof}
For $X=\Sigma$, the definitions of the face pullback \eqref{eq:def-push-face-x}, the trace \eqref{eq:defTrsandw}, the element pullback \eqref{eq:defPsandRs} and the extension \eqref{eq:def_gi_ext} show
$$
(\mathcal R_\phi^\Sigma g_1)(\bx) = (\TrsFz\mathcal R_{\Phi}^\Sigma g_1^e)(\bx) = \mathcal R_{\Phi}^\Sigma g_1^e(\bx) = g_1^e(\Phi(\bx)) =
g_1(\phi(\bx)) \quad \text{ for } \bx\in F_0.
$$
Similarly, $\mathcal P_\phi^\Sigma g_0=g_0\circ\phi^{-1}$, which proves \eqref{eq:def-face-s-pull}.

For $X=W$, fix $\bx\in F_0$ and let $\btau$ be any
vector tangent to $F_0$ at $\bx$. 
It follows from \eqref{eq:def-push-face-x},  \eqref{eq:defTrsandw} and \eqref{eq:defPwandRw} that 
\begin{equation}\label{eq:Rwphi0-1}
\begin{aligned}
(\mathcal R_\phi^Wg_1)(\bx)\cdot\btau = \bigl( \bP_{F_0}(\bx)D\Phi(\bx)^\trans g_1^e(\phi(\bx))
\bigr)\cdot\btau = g_1^e(\phi(\bx))\cdot D\Phi(\bx)\btau,
\end{aligned}
\end{equation}
where $\bP_{F_0}(\bx)\btau=\btau$ is used.
Recall from  Section~\ref{subsec:facemapping} that $\Dtau\phi=(D\Phi|_{F_0})\bP_{F_0}$. Then $D\Phi(\bx)\btau=\Dtau\phi(\bx)\btau$,
and this vector is tangent to $F_1$ at $\phi(\bx)$.
Since $g_1=\TrwFo g_1^e$, it follows that
$$
\begin{aligned}
g_1^e(\phi(\bx))\cdot D\Phi(\bx)\btau
&= g_1(\phi(\bx))\cdot\Dtau\phi(\bx)\btau =
\bigl( (\Dtau\phi(\bx))^\trans g_1(\phi(\bx)) \bigr)\cdot\btau.
\end{aligned}
$$
Since $(\Dtau\phi(\bx))^\trans = \bP_{F_0}(\bx)D\Phi(\bx)^\trans$, $(\Dtau\phi(\bx))^\trans g_1(\phi(\bx))$ is tangent to $F_0$ at $\bx$. Both vectors are therefore tangential and have the same
inner product with every tangential vector $\btau$.
This and \eqref{eq:Rwphi0-1} prove the first identity in
\eqref{eq:def-face-w-pull}.
Applying the same argument to $\Phi^{-1}$, whose
restriction to $F_1$ is $\phi^{-1}$, proves the second
identity in \eqref{eq:def-face-w-pull}.
All expressions in
\eqref{eq:def-face-s-pull}--\eqref{eq:def-face-w-pull}
depend only on the face data and $\phi$.
Hence the face transformations are independent of
both the chosen extensions and the inducing element
diffeomorphism.
\end{proof}

\begin{corollary}\label{cor-facetransofrmnation}
The face transformations defined in \eqref{eq:def-push-face-x} satisfy
$\mathcal R_\phi^X=(\mathcal P_\phi^X)^{-1}$. Moreover,
\begin{subequations}
\begin{align}
\mathcal{P}^{X}_{\phi}\TrxFz&=\TrxFo\mathcal{P}^{X}_{\Phi},\quad\mathcal{R}^{X}_{\phi}\TrxFo=\TrxFz\mathcal{R}^{X}_{\Phi},\qquad&\text{(Trace compatibility)}\label{eq:face-commuit-push-x}\\
\mathcal{P}^{X}_{\psi\circ \phi} &= \mathcal{P}^{X}_{\psi}\mathcal{P}^{X}_{\phi}, \quad
\mathcal{R}^{X}_{\psi\circ \phi} = \mathcal{R}^{X}_{\phi}\mathcal{R}^{X}_{\psi},\qquad&\text{(Composition)}\label{eq:face-composit-push-x}
\end{align}
\end{subequations}
where $\psi:F_1\to F_2$ is induced by an element diffeomorphism $\Psi:T_1\to T_2$.
\end{corollary}
\begin{proof}
Let $X\in\{\Sigma,W\}$.
For any sufficiently smooth function $v$ on $T_0$,
$v$ itself is an admissible extension of
$\TrxFz v$. Thus, \eqref{eq:def-push-face-x} gives $\mathcal P_\phi^X(\TrxFz v)= \TrxFo(\mathcal P_\Phi^Xv)$.
Applying the same argument to a function on $T_1$
proves the pullback identity in
\eqref{eq:face-commuit-push-x}.
To prove the inverse property, let $g_i$ be a trace
function on $F_i$, and choose an extension $g_i^e$
satisfying $\TrxFi g_i^e=g_i$, for $i=0,1$.
By the trace compatibility \eqref{eq:face-commuit-push-x} and the inverse properties
of the element transformations,
$$
\mathcal R_\phi^X\mathcal P_\phi^Xg_0
= \mathcal R_\phi^X\TrxFo (\mathcal P_\Phi^Xg_0^e) =
\TrxFz (\mathcal R_\Phi^X\mathcal P_\Phi^Xg_0^e)
= \TrxFz g_0^e = g_0.
$$
Similarly, one can obtain  $\mathcal P_\phi^X\mathcal R_\phi^Xg_1 = \TrxFo g_1^e = g_1$. Hence $\mathcal R_\phi^X=(\mathcal P_\phi^X)^{-1}$.
Finally, $\Psi \circ\Phi$ induces $\psi\circ\phi$
on $F_0$.
The composition property of the element
transformations and the trace compatibility \eqref{eq:face-commuit-push-x} yield
$$
\begin{aligned}
\mathcal P_{\psi\circ\phi}^Xg_0 =
\TrxFt(\mathcal P_{\Psi \circ\Phi}^Xg_0^e) =
\TrxFt(\mathcal P_\Psi^X\mathcal P_\Phi^Xg_0^e) =
\mathcal P_\psi^X \TrxFo(\mathcal P_\Phi^Xg_0^e) =
\mathcal P_\psi^X\mathcal P_\phi^Xg_0.
\end{aligned}
$$
Taking inverses gives $\mathcal R_{\psi\circ\phi}^X =
(\mathcal P_\psi^X\mathcal P_\phi^X)^{-1} =
\mathcal R_\phi^X\mathcal R_\psi^X$, which completes the proof.
\end{proof}

Consider an interior face $F:=\partial T^+\cap\partial T^-$. Let $\rT^\pm:=\Theta_h(T^\pm)$ and $\rF:=\Theta_h(F)$ be the exact curved elements and their common face, respectively. Set $$\Fhkpm:=\FOhkpm:\hT\to T^\pm,\qquad\Fhrpm:=\FOhrpm:\hT\to\rT^\pm.$$
Let $\hF^\pm:=(\Fhkpm)^{-1}(F)\subset\partial\hT$. Denote the corresponding face mappings by
\begin{equation}\label{eq:def-face-mapping}
\FFhkpm:=\Fhkpm|_{\hF^\pm}:\hF^\pm\to F,\qquad\FFhrpm:=\Fhrpm|_{\hF^\pm}:\hF^\pm\to \rF.   
\end{equation}
Define the affine mapping $\gamma_F:= (\FFhkm)^{-1} \FFhkp: \hF^+\to\hF^-$, then $\FFhkp=\FFhkm\circ\gamma_F$. Moreover, the compatibility of the mapping $\Theta_h$ across $F$ gives $\FFhrp=\FFhrm\circ\gamma_F$.
The face map $\gamma_F$ extends to an affine bijection $\gamma:\hat T\to\hat T$ that maps the vertex opposite $\hat F^+$ to that opposite $\hat F^-$, i.e., $\gamma_F=\gamma|_{\hF^+}$.
The composition property in Corollary~\ref{cor-facetransofrmnation} therefore applies and yields
\begin{equation}\label{eq:facemapgamma}
\PFhkxp=\mathcal{P}^X_{\FFhkm\circ\gamma_F}=\PFhkxm\PFhhx,\qquad\RFhrxp=\mathcal{R}^X_{\FFhrm\circ\gamma_F}=\RFhhx\RFhrxm.
\end{equation}
The first identity, together with the inverse property in Corollary~\ref{cor-facetransofrmnation}, gives
\begin{equation}
\label{eq:facemapgamma2}
\PFhkxm = \PFhkxp (\PFhhx)^{-1} = \PFhkxp \RFhhx.
\end{equation}

\begin{theorem}\label{thm:abstract-pi-conformity}
For each face $\hF$ of the reference tetrahedron $\hT$, assume that there is a face interpolation operator $\pixhF$ satisfying
\begin{subequations}\label{eq:faceInterpolation-x-assmue}
\begin{align}
\pixhF\TrxhF&=\TrxhF\pixhT.\qquad&\text{(Trace compatibility)}\label{eq:commuit-pix}\\
\pixhFp\RFhhx&=\RFhhx\pixhFm.\qquad&\text{(Pullback compatibility)}\label{eq:commuit-pixhFTrx}
\end{align}
\end{subequations}
Then the elementwise interpolants have matching traces across any interior face $F$:
\begin{equation}\label{eq:conforminginterpolation}
\TrxF(\pixtp v^+)=\TrxF(\pixtm v^-),\qquad v^\pm= v|_{\rT^\pm}.
\end{equation}
\end{theorem}
\begin{proof}
Fix an interior face $F \in \calFhI$ and its corresponding exact
face $\rF = \Fkrh(F)$. Since $v$ has a single trace on $\rF$, define
$$
g:=\TrxrF v^+=\TrxrF v^-,
\qquad
\hat v^\pm:=\Rhrxpm v^\pm,
\qquad
\hat g^\pm:=\RFhrxpm g.
$$
By the trace compatibility of the pullback \eqref{eq:face-commuit-push-x}, $\TrxhFpm\hat v^\pm = \RFhrxpm\TrxrF v^\pm = \hat g^\pm$.
Then the definition of the local element interpolation operators in \eqref{def:local-projectors}, trace compatibility of the pushforward in \eqref{eq:face-commuit-push-x}, and the trace compatibility in \eqref{eq:commuit-pix} give
\begin{equation}\label{eq:face-factorization}
\begin{aligned}
\TrxF(\pixtpm v^\pm)
&=\TrxF\bigl(\Phkxpm\pixhT\hat v^\pm\bigr) =\PFhkxpm\TrxhFpm\bigl(\pixhT\hat v^\pm\bigr)\\
&=\PFhkxpm\pixhFpm\bigl(\TrxhFpm\hat v^\pm\bigr) =\PFhkxpm\pixhFpm\hat g^\pm.
\end{aligned}
\end{equation}
Thus, the trace of each element interpolant is determined by the corresponding face interpolant, as illustrated in Figure~\ref{fig:global-trace-compatibility}(a).
The second face transformation identity in
\eqref{eq:facemapgamma}, the pullback compatibility in \eqref{eq:commuit-pixhFTrx}, and \eqref{eq:facemapgamma2} show
$$
\PFhkxp\pixhFp\hat g^+  = \PFhkxp \pixhFp \RFhhx \hat g^- =\PFhkxp\RFhhx\pixhFm\hat g^- =\PFhkxm\pixhFm\hat g^-.
$$
By \eqref{eq:face-factorization}, this implies that
$\TrxF(\pixtp v^+) = \TrxF(\pixtm v^-)$. 
These relations are illustrated in
Figure~\ref{fig:global-trace-compatibility}(b).
Since $F$ is arbitrary, the proof is complete. 
\end{proof}

\begin{figure}[htbp]
\centering
\begin{minipage}[t]{0.49\textwidth}
\centering
\scriptsize
\begin{tikzcd}[
  row sep=2.8em,
  column sep=2.2em
]
v^\pm
  \arrow[r, "\TrxrF"]
  \arrow[d, "\Rhrxpm"']
&
g
  \arrow[d, "\RFhrxpm"]
\\
\hat v^\pm
  \arrow[r, "\TrxhFpm"]
  \arrow[d, "\pixhT"']
&
\hat g^\pm
  \arrow[d, "\pixhFpm"]
\\
\pixhT\hat v^\pm
  \arrow[r, "\TrxhFpm"]
  \arrow[d, "\Phkxpm"']
&
\pixhFpm\hat g^\pm
  \arrow[d, "\PFhkxpm"]
\\
\pixtpm v^\pm
  \arrow[r, "\TrxF"]
&
\TrxF(\pixtpm v^\pm)
\end{tikzcd}

\smallskip
{\small (a) From element interpolation to face interpolation.}
\end{minipage}
\hfill
\begin{minipage}[t]{0.49\textwidth}
\centering
\scriptsize
\begin{tikzcd}[
  row sep=2.8em,
  column sep=0.8em
]
&
g
  \arrow[dl, "\RFhrxp"']
  \arrow[dr, "\RFhrxm"]
&
\\
\hat g^+
  \arrow[d, "\pixhFp"']
&&
\hat g^-
  \arrow[ll, "\RFhhx"']
  \arrow[d, "\pixhFm"]
\\
\pixhFp\hat g^+
  \arrow[d, "\PFhkxp"']
&&
\pixhFm\hat g^-
  \arrow[ll, "\RFhhx"']
  \arrow[d, "\PFhkxm"]
\\
\TrxF(\pixtp v^+)
  \arrow[rr, equal]
&&
\TrxF(\pixtm v^-)
\end{tikzcd}

\smallskip
{\small (b) Matching traces across the common face.}
\end{minipage}
\caption{Compatibility of the interpolated traces.}
\label{fig:global-trace-compatibility}
\end{figure}
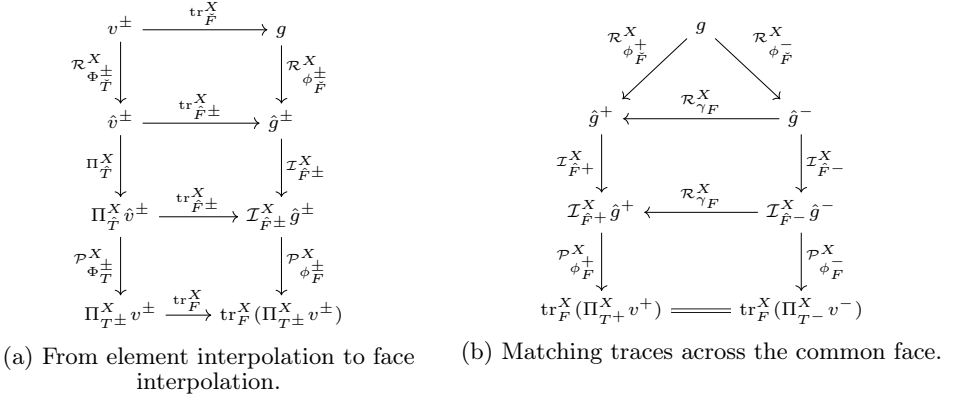

\begin{remark}
The local commuting identity \eqref{eq:local-commuting-checked}, together with the global conformity established above, yields a pair of globally conforming commuting interpolation operators on curved meshes. This construction is not specific to the present pair: it extends to other reference commuting operators whenever the corresponding transformations preserve the commuting identities and the face interpolation operators satisfy the required compatibility conditions.
\end{remark}

\subsubsection{The conformity for \texorpdfstring{$\pis$}{Pi Sigma h}}\label{sec:verXbeS}
According to Theorem~\ref{thm:abstract-pi-conformity}, it remains to construct a face interpolation operator satisfying \eqref{eq:faceInterpolation-x-assmue}. Such an operator is obtained from the DOFs defining $\hpis$ in \cite[Lemma~4.8]{fu2020exact} that determine the trace on a face.

Given a face $\hF$ of $\hT$, let $\mathcal{V}(\hF)$ and $\mathcal{E}(\hF)$ denote the sets of vertices and edges of $\hF$, respectively, and fix two orthonormal tangent vectors $\hat\btau_1$ and $\hat\btau_2$ to $\hF$. 
For $\hat e \in \mathcal{E}(\hF)$, let $\hat\bn_{\hat e,\hF}$ denote the unit vector parallel to $\hF$ and outward normal to $\hat e$.

\begin{lemma}\label{lem:trace-interpolation}
Set $r:=k+2 \geq 5$.
A function $q\in P_r(\hF)$ is uniquely determined by the following DOFs
\begin{subequations}\label{eq:trace-dofs}
\begin{align}
&\partial_{\hat\btau_1}^{\alpha_1}
 \partial_{\hat\btau_2}^{\alpha_2}q(\hat a)
&& \text{ for all } 0\leq\alpha_1+\alpha_2\le2, 
\, \hat a\in\mathcal{V}(\hF),
\label{eq:trace-dofs-a}\\
&(q,\psi)_{\hat e}
&& \text{ for all } \psi\in P_{r-6}(\hat e),\, \hat e\in\mathcal{E}(\hF),
\label{eq:trace-dofs-b}\\
&(\partial_{\hat\bn_{\hat e,\hF}}q,\eta)_{\hat e}
&& \text{ for all } \eta\in P_{r-5}(\hat e), \, \hat e\in\mathcal{E}(\hF),
\label{eq:trace-dofs-c}\\
&(q,\psi)_{\hF} && \text{ for all } \psi\in P_{r-6}(\hF).
\label{eq:trace-dofs-d}
\end{align}
\end{subequations}
Let $\pishF:H^{7/2}(\hF)\to P_r(\hF)$ denote the corresponding
interpolation operator. Then
\begin{subequations}\label{eq:interfacesproperty}
\begin{align}
\TrshF(\hpis\hat\sigma)&=\pishF(\TrhF\hat\sigma),\qquad\quad\,\, \text{ for all }  \hat\sigma\in H^4(\hT),\label{eq:trace-interpolation}\\
\pishFp(\RFhhs \hat{g}^-)&=\RFhhs (\pishFm \hat{g}^-), \qquad \text{ for all }  \hat{g}^-\in H^{7/2}(\hF^-).\label{eq:trace-affine-covariance}
\end{align}    
\end{subequations}
\end{lemma}

\begin{proof}
A dimension count shows the number of DOFs in \eqref{eq:trace-dofs} equals $\dim P_r(\hF)$. Suppose that $q\in P_r(\hF)$ vanishes on all these DOFs. Then \eqref{eq:trace-dofs-a}--\eqref{eq:trace-dofs-c} imply $q|_{\hat e}=0$ and $\partial_{\hat\bn_{\hat e,\hF}}q|_{\hat e}=0$ hold for all $\hat e\in \mathcal{E}(\hF)$. Hence, $q=(\hat\lambda_1\hat\lambda_2\hat\lambda_3)^2p$ for some $p\in P_{r-6}(\hF)$, where $\hat\lambda_1,\hat\lambda_2,\hat\lambda_3$ are the barycentric coordinates on $\hF$. Taking $\psi=p$ in \eqref{eq:trace-dofs-d} gives $0=(q,p)_{\hF}=\big((\hat\lambda_1\hat\lambda_2\hat\lambda_3)^2p,p\big)_{\hF}$, and hence $p=0$. Thus $q=0$, which proves the unisolvence of \eqref{eq:trace-dofs}.

For each face DOF $\ell$ in \eqref{eq:trace-dofs}, define $\widetilde\ell(\hat\sigma):=\ell(\TrshF\hat\sigma)$. By \cite[Lemma~4.8]{fu2020exact}, $\widetilde\ell$ is a linear combination of the element DOFs $\ell_{\hT}$ defining $\hpis$. Since $\hpis$ preserves these element DOFs: $\ell_{\hT}(\hpis\hat\sigma)=\ell_{\hT}(\hat\sigma)$,  this gives $\widetilde\ell(\hpis\hat\sigma)=\widetilde\ell(\hat\sigma)$. Therefore,
$$
\ell\big(\TrshF(\hpis\hat\sigma)\big)
=\widetilde{\ell}(\hpis\hat\sigma)=\widetilde{\ell}(\hat\sigma)=\ell\big(\TrshF(\hat\sigma)\big)=\ell\big(\pishF(\TrshF\hat\sigma)\big).
$$
Since both sides of \eqref{eq:trace-interpolation} belong to
$P_r(\hF)$, the unisolvence of \eqref{eq:trace-dofs} yields \eqref{eq:trace-interpolation}.

For any DOF $\ell^+$ in \eqref{eq:trace-dofs} on $\hF^+$, define $\ell_{\gamma_F}^-(\hat{g}^-):=\ell^+(\RFhhs \hat{g}^-)$. 
Claim that $\ell_{\gamma_F}^-$ is a linear combination of the DOFs $\ell^-$ in \eqref{eq:trace-dofs} on $\hF^-$.
Since $\pishFm$ preserves every DOF $\ell^-$: $\ell^-(\pishFm\hat g^-)=\ell^-(\hat g^-)$, the linearity of $\ell_{\gamma_F}^-$ then gives $\ell_{\gamma_F}^-(\pishFm\hat g^-)=\ell_{\gamma_F}^-(\hat g^-)$.
Consequently,
\begin{equation*}
\ell^+\bigl(\RFhhs(\pishFm\hat g^-)\bigr)
=\ell_{\gamma_F}^-(\pishFm\hat g^-)
=\ell_{\gamma_F}^-(\hat g^-)
=\ell^+(\RFhhs\hat g^-)
=\ell^+\bigl(\pishFp(\RFhhs\hat g^-)\bigr),    
\end{equation*}
where the last equality follows from the definition of
$\pishFp$. Since $\gamma_F$ is affine, both $\pishFp(\RFhhs\hat g^-)$ and $\RFhhs(\pishFm\hat g^-)$ belong to $P_r(\hF^+)$.
The unisolvence of \eqref{eq:trace-dofs} on $\hF^+$ therefore yields \eqref{eq:trace-affine-covariance}.

It remains to show that $\ell_{\gamma_F}^-$ is a linear combination of the DOFs in \eqref{eq:trace-dofs} on $\hF^-$. This follows from the chain rule for \eqref{eq:trace-dofs-a} and from a change of variables for \eqref{eq:trace-dofs-b} and \eqref{eq:trace-dofs-d}.
For a DOF in \eqref{eq:trace-dofs-c} associated with an edge $\hat e^+\in\mathcal{E}(\hF^+)$, let
$\hat e^-=\gamma_F(\hat e^+)$ and let~$\hat \bt^-$ be a unit tangent vector to $\hat e^-$.
Since $\gamma_F$ is affine,
$\partial_{\hat\bn_{\hat e^+,\hF^+}}(\RFhhs\hat g^-)
=\big(c_1 \,\partial_{\hat \bt^-}\hat g^-
+c_2\,\partial_{\hat\bn_{\hat e^-,\hF^-}}\hat g^-
\big)\circ\gamma_F$ for some constants $c_1$ and $c_2$. Thus, after a change of variables, $\ell_{\gamma_F}^-(\hat g^-)$ is a linear combination of
$$(\partial_{\hat \bt^-}\hat g^-,\eta^-)_{\hat e^-}\quad\text{and}\quad(\partial_{\hat\bn_{\hat e^-,\hF^-}}\hat g^-,\eta^-)_{\hat e^-},\qquad \eta^-\in P_{r-5}(\hat e^-),$$
where the second term is a DOF in \eqref{eq:trace-dofs-c}. For the first term, let $\hat a_1^-$ and $\hat a_2^-$ be the endpoints of $\hat e^-$. Integration by parts gives
$
(\partial_{\hat \bt^-}\hat g^-,\eta^-)_{\hat e^-} = \left.\hat g^-\eta^-\right|_{\hat a_1^-}^{\hat a_2^-} - (\hat g^-,\partial_{\hat \bt^-}\eta^-)_{\hat e^-}
$, which is a linear combination of the DOFs in \eqref{eq:trace-dofs-a} and \eqref{eq:trace-dofs-b} since $\partial_{\hat \bt^-}\eta^-\in P_{r-6}(\hat e^-)$.
Thus $\ell_{\gamma_F}^-$ is a linear combination of the DOFs in
\eqref{eq:trace-dofs} on $\hF^-$, which completes the proof.
\end{proof}

The identities established in Lemma~\ref{lem:trace-interpolation} verify the hypotheses of Theorem~\ref{thm:abstract-pi-conformity}
for $X=\Sigma$. Therefore, $\pis\sigma\in H^1(\Omega_h)$ for every $\sigma\in H^4(\Omega)$.

\subsubsection{The conformity for \texorpdfstring{$\piw$}{Pi W h}}
\label{sec:verXbeW}

To apply Theorem~\ref{thm:abstract-pi-conformity}, this subsection constructs a face interpolation operator satisfying \eqref{eq:faceInterpolation-x-assmue}, which is obtained from the DOFs defining $\hpiw$ in \cite[Lemma~4.11]{fu2020exact} that determine the tangential trace on a face.

For $s\geq0$, set $\bP_{s}^{\,\mathrm{tan}}(\hat F) :=
\{ \hbw\in \bP_{s}(\hat F): \hbw\cdot\hat{\boldsymbol n}_{\hF}=0\}$ and $\bP_{s}^{\,\mathrm{tan}}(\hat e; \hat F):= \{ \hbw\in \bP_{s}(\hat e): \hbw\cdot\hat{\boldsymbol n}_{\hF}=0\}$.
For smooth functions $\hat q$ and tangential vector field $\hat\bw$ on $\hF$, define $\grad_{\hF}\hat q:=\nabla_{\hF}\hat q$, $\rot_{\hF}\hat q:=\nabla_{\hF}\hat q\times\hat\bn_{\hF}$, $\Div_{\hF}\hbw:=\operatorname{tr}(\nabla_{\hF}\hbw)$ and
$\curl_{\hF}\hbw:=\Div_{\hF}(\hbw\times\hat\bn_{\hF})$.
Since $\hF$ is planar, these definitions agree with those in~\cite{fu2020exact}. Stokes theorem on $\hF$ yields
\begin{equation}\label{eq:surface-Stokes}
(\curl_{\hF}\hat{\bw},\hat q)_{\hF} - (\hat{\bw},\rot_{\hF}\hat q)_{\hF} = (\hat{\bw}\cdot\hat{\bt},\hat q)_{\partial\hF},    
\end{equation}
where $\hat\bt$ is the unit tangent vector along $\partial\hF$ with the induced orientation. Define $\boldsymbol{\mathcal D}_{s}(\hF)\coloneqq \bP_{s-1}^{\,\mathrm{tan}}(\hF)+(\boldsymbol{\hx}-\boldsymbol{\hx}_{\hF}^c)P_{s-1}(\hF)$ as the local Raviart--Thomas space on $\hF$, with $\boldsymbol{\hx}_{\hF}^c$ the barycenter of $\hF$.
Define $\bH^{5/2}_{\tan}(\hF):=\{\hat\bg\in\bH^{5/2}(\hF):\hat\bg\cdot\hbn_{\hF}=0\}$.

\begin{lemma}
\label{lem:W-face-dofs} 
Set $r:=k+2 \geq 5$. A function $\hbw\in \bP_{r-1}^{\,\mathrm{tan}}(\hF)$ is uniquely determined by the following DOFs:
\begin{subequations}
\label{eq:W-face-dofs}
\begin{align}
& \partial_{\hat{\boldsymbol\tau}_1}^{\alpha_1}
\partial_{\hat{\boldsymbol\tau}_2}^{\alpha_2}
(\hbw\cdot\hat{\boldsymbol\tau}_j)(\hat a),
&& \text{ for all } 0\leq\alpha_1+\alpha_2\le1,j=1,2,\,\hat a\in\mathcal V(\hF),
\label{eq:W-face-dofs:a} \\
& (\hbw,\hat{\boldsymbol\kappa})_{\hat e},
&& \text{ for all } \hat{\boldsymbol\kappa}\in
\bP_{r-5}^{\,\mathrm{tan}}(\hat e;\hF), \, \hat e\in\mathcal E(\hF),
\label{eq:W-face-dofs:b} \\
& (\curl_{\hF}\hbw,\hat\eta)_{\hat e}, && \text{ for all }
\hat\eta\in P_{r-4}(\hat e), \,
\hat e\in\mathcal E(\hF),
\label{eq:W-face-dofs:c} \\
& (\hbw,\hat{\boldsymbol\kappa})_{\hF},
&& \text{ for all } \hat{\boldsymbol\kappa}\in\boldsymbol{\mathcal D}_{r-5}(\hF).
\label{eq:W-face-dofs:d}
\end{align}
\end{subequations}
Let $\piwhF: \bH^{5/2}_{\tan}(\hF)\to \bP_{r-1}^{\,\mathrm{tan}}(\hF)$ be the interpolation operator of \eqref{eq:W-face-dofs}. Then
\begin{subequations}
\begin{align}
\TrwhF(\hpiw\hbw)
&=\piwhF(\TrwhF\hbw),\qquad\,\,\text{ for all } 
\hbw \in\bH^3(\hat T), \label{eq:W-face-interpolation}\\
\piwhFp(\RFhhw\hat\bg^-)&= \RFhhw(\piwhFm\hat\bg^-),\qquad\text{for all }~\hat\bg^-\in\bH^{5/2}_{\tan}(\hF^-).\label{eq:W-face-affine}
\end{align} 
\end{subequations}
\end{lemma}

\begin{proof}
A direct dimension count shows that the number of DOFs in
\eqref{eq:W-face-dofs} equals
$\dim \bP_{r-1}^{\,\mathrm{tan}}(\hF).$
Suppose that $\hbw\in\bP_{r-1}^{\,\mathrm{tan}}(\hF)$ vanishes on all these DOFs. The vertex and edge DOFs \eqref{eq:W-face-dofs:a}--\eqref{eq:W-face-dofs:b} imply
$\hbw|_{\partial\hF}=0.$
Set $\hat q:=\curl_{\hF}\hbw\in P_{r-2}(\hF)$.
The DOFs \eqref{eq:W-face-dofs:a} and \eqref{eq:W-face-dofs:c} imply $\hat q|_{\partial\hF}=0$, and hence $\hat q=\hat\lambda_1\hat\lambda_2\hat\lambda_3\,\hat p$
for some $\hat p\in P_{r-5}(\hF).$
Since $\rot_{\hF}\hat p\in \boldsymbol{\mathcal D}_{r-5}(\hF)$,
the Stokes formula \eqref{eq:surface-Stokes}, $\hat\bw|_{\partial\hF}=0$, and \eqref{eq:W-face-dofs:d} yield
$$0=(\hbw,\rot_{\hF}\hat p)_{\hF}=(\hat q,\hat p)_{\hF}=(\hat\lambda_1\hat\lambda_2\hat\lambda_3\hat p,\hat p)_{\hF}.$$
Thus $\hat p=\hat q=\curl_{\hF}\hbw=0$.
It follows that $\hbw=\operatorname{grad}_{\hF}\hat{q}_0 $
for some $\hat{q}_0\in P_r(\hF)$. Since $\hbw|_{\partial \hF}=0$, after subtracting a constant, $\hat{q}_0=(\hat\lambda_1\hat\lambda_2\hat\lambda_3)^2\hat{p}_0$
for some $\hat{p}_0\in P_{r-6}(\hF)$. For any $\hat{\boldsymbol\kappa}\in \boldsymbol{\mathcal D}_{r-5}(\hF)$, \eqref{eq:W-face-dofs:d} and integration by parts yield
$$
0=(\operatorname{grad}_{\hF}\hat q_0,\hat{\boldsymbol\kappa})_{\hF}=-(\hat{q}_0,\Div_{\hF}\hat{\boldsymbol\kappa})_{\hF}=-((\hat\lambda_1\hat\lambda_2\hat\lambda_3)^2\hat p_0,\Div_{\hF}\hat{\boldsymbol\kappa})_{\hF}.
$$
Since $\Div_{\hF} \boldsymbol{\mathcal D}_{r-5}(\hF) =P_{r-6}(\hF)$, taking $\Div_{\hF}\hat{\boldsymbol\kappa}=\hat{p}_0$ gives $\hat p_0=\hat q_0=0$, and therefore $\hbw=0$. This proves unisolvence.

For each DOF $\ell$ in \eqref{eq:W-face-dofs}, set $\widetilde\ell(\hbv):=\ell(\TrwhF\hbv)$. Since $\widetilde\ell$ is a linear combination of the DOFs in \cite[Lemma~4.11]{fu2020exact} preserved by $\hpiw$, similar arguments for \eqref{eq:trace-interpolation} apply to \eqref{eq:W-face-interpolation}.  

For a DOF $\ell^+$ in \eqref{eq:W-face-dofs} on $\hF^+$, set $\ell_{\gamma_F}^-(\boldsymbol g):=\ell^+(\RFhhw\boldsymbol g)$. As illustrated in the proof of \eqref{eq:trace-affine-covariance}, it suffices to show that $\ell_{\gamma_F}^-$ lies in the span of the DOFs in \eqref{eq:W-face-dofs} on $\hF^-$. The cases \eqref{eq:W-face-dofs:a}--\eqref{eq:W-face-dofs:b} follow from the chain rule and an affine change of variables, respectively.
For \eqref{eq:W-face-dofs:c}, a direct computation shows that $\curl_{\hF^+}\bigl(\RFhhw\bg\bigr)$ and $\mu_{\gamma_F}(\curl_{\hF^-}\bg)\circ\gamma_F$ agree up to a sign determined by the face orientations. This and $\mu_{\gamma_F}$ being a constant implies that these edge DOFs transform into DOFs of the same type on $\hF^-$.
For \eqref{eq:W-face-dofs:d}, a change of variables gives 
$$(\RFhhw\boldsymbol g,\hat{\boldsymbol\kappa}^+)_{\hF^+}=(\boldsymbol g,\hat{\boldsymbol\kappa}^-)_{\hF^-},\qquad \hat{\boldsymbol\kappa}^-:=\mu_{\gamma_F}^{-1}\Dtau\gamma_F(\hat{\boldsymbol\kappa}^+\circ\gamma_F^{-1}).$$
Since $\gamma_F$ maps the barycenter of $\hF^+$ to that of
$\hF^-$, for every $\hx^-\in\hF^-$, it holds that $\Dtau\gamma_F\bigl(\gamma_F^{-1}(\hx^-)-\hx_{\hF^+}^c\bigr)=\hx^--\hx_{\hF^-}^c$.
Moreover, $\Dtau\gamma_F$ maps tangential vectors on $\hF^+$ to tangential vectors on $\hF^-$. The definition of $\boldsymbol{\mathcal D}_{r-5}$ therefore gives $\hat{\boldsymbol\kappa}^-\in\boldsymbol{\mathcal D}_{r-5}(\hF^-)$.
Therefore, $\ell_{\gamma_F}^-$ lies in the span of the DOFs in \eqref{eq:W-face-dofs} on $\hF^-$, which completes the proof.
\end{proof}

The identities established in Lemma~\ref{lem:W-face-dofs} verify the hypotheses of Theorem~\ref{thm:abstract-pi-conformity}
for $X=\bW$. Therefore, $\piw \bw \in \bH( \curl; \Omega_h)$ for every $\bw \in \bH^3(\Omega)$.

\subsection{Verification of \texorpdfstring{\eqref{eq:abstract-approximation}}{the abstract approximation property}}
Let $\bv \in \bH^k(\Omega)$ with $k \ge 3$, and denote by $\bv$ also its bounded extension to $\mathbb{R}^3$.
In the following, $\Rhkw$ acts on the restriction of this extension to $T \in \calTh$, whereas $\Rhrw$ acts on $\bv|_{\rT}$ with $\rT = \Fkr(T)$.

For each $T\in\calTh$, the inverse relation between
$\Phkw$ and $\Rhkw$, together with the definition of the element interpolation in \eqref{def:local-projectors}, gives $\bv=\Phkw\Rhkw\bv$ and $\piwt\bv=\Phkw\hpiw\Rhrw\bv$ on $T$.
Adding and subtracting $\Phkw\Rhrw\bv$ yields
\begin{equation}\label{eq:sum}
\bv-\piwt\bv = \Phkw(\Rhkw\bv-\Rhrw\bv) + \Phkw(I-\hpiw)\Rhrw\bv.
\end{equation}
The first term measures the geometric error,
and the second is the reference interpolation error
pushed forward to $T$.

For the geometric error, the chain rule with $\Fhr =\Fkr \circ \Fhk$ and the definitions of the covariant transformation in \eqref{eq:defPwandRw} imply $\Phkw\Rhrw\bv = \Rkrw \bv =D\Fkr^\trans(\bv\circ\Fkr)$. Hence,
$$
\Phkw(\Rhkw\bv-\Rhrw\bv) = (I-D\Fkr^\trans)(\bv\circ\Fkr)
+ \bv-\bv\circ\Fkr \quad  \text{ on } T.
$$
By \eqref{eq:psi-estimates-1} with $m=1$ and $m=0$,
respectively, $\|D\Fkr-I\|_{L^\infty(T)} \lesssim h_T^k$ and 
$\|\Fkr-\operatorname{id}_T\|_{L^\infty(T)}
\lesssim h_T^{k+1}$.
Since $|T|^{1/2}\lesssim h_T^{3/2}$ by
\eqref{Prop-FT}, it holds that
$$
\|(I-D\Fkr^\trans)(\bv\circ\Fkr)\|_{L^2(T)}
\leq |T|^{1/2} \|I-D\Fkr^\trans\|_{L^\infty(T)} \|\bv\circ\Fkr\|_{L^\infty(T)} \lesssim
h_T^{k+3/2}\|\bv\|_{L^\infty(\mathbb R^3)}.
$$
The mean-value formula, applied componentwise, leads to $|\bv(x)-\bv(\Fkr(x))|\leq\|D\bv\|_{L^\infty(\mathbb R^3)}|x-\Fkr(x)|$ for all $x\in T$.
Consequently,
$$
\|\bv-\bv\circ\Fkr\|_{L^2(T)}
\leq |T|^{1/2}\|D\bv\|_{L^\infty(\mathbb R^3)} \|\operatorname{id}_T-\Fkr\|_{L^\infty(T)} \lesssim h_T^{k+5/2}\|D\bv\|_{L^\infty(\mathbb R^3)}.
$$
Combining the two bounds, using
$H^3(\mathbb R^3)\hookrightarrow W^{1,\infty}(\mathbb R^3)$
and the stability of the extension, yields
\begin{equation}
\label{eq:sum-1}
\|\Phkw(\Rhkw\bv-\Rhrw\bv)\|_{L^2(T)} \lesssim
h_T^{k+3/2}\|\bv\|_{H^3(\Omega)}.
\end{equation}

For the interpolation term, set $\hbv_T:=\Rhrw\bv = D\Fhr^\trans(\bv\circ\Fhr)$. A scaling argument analogous to that in Lemma~\ref{lem::normH1XThXT} gives
$$
\|\Phkw(I-\hpiw)\Rhrw\bv\|_{L^2(T)}=\|\Phkw(I-\hpiw)\hbv_T\|_{L^2(T)} \lesssim h_T^{1/2} \|\hbv_T-\hpiw\hbv_T\|_{L^2(\hT)}.
$$
Since the DOFs defining $\hpiw$ are bounded on $\bH^3(\hT)$, and $\hpiw$ reproduces $\bP_{{k+1}}(\hT)$ with $k\geq3$, the Bramble--Hilbert lemma yields
$$
\|\hbv_T-\hpiw\hbv_T\|_{L^2(\hT)}\lesssim\inf_{\hbw\in\bP_{k-1}(\hT)}
\|\hbv_T-\hbw\|_{H^3(\hT)} \lesssim |\hbv_T|_{H^k(\hT)}.$$
The pullback estimate in Lemma~\ref{lem::NormInThT} applies also to $\Fhr$ by
\eqref{eq:theta-estimates}.
Thus, Leibniz' rule and
\eqref{eq:theta-estimates-1} give
$$
\begin{aligned}
|\hbv_T|_{H^k(\hT)} &\lesssim
\sum_{j=0}^{k} \|D^{j+1}\Fhr\|_{L^\infty(\hT)} |\bv\circ\Fhr|_{H^{k-j}(\hT)}
\\ &\lesssim \sum_{j=0}^{k} h_T^{j+1}h_T^{k-j-3/2}
\|\bv\|_{H^{k-j}(\rT)} \lesssim h_T^{k-1/2}\|\bv\|_{H^k(\rT)}.
\end{aligned}
$$
This shows that
\begin{equation}\label{eq:sum-2}
\|\Phkw(I-\hpiw)\Rhrw\bv\|_{L^2(T)} \lesssim h_T^k\|\bv\|_{H^k(\rT)}.
\end{equation}

Combining \eqref{eq:sum}--\eqref{eq:sum-2} and summing
the squared local estimates yields
$$
\begin{aligned}
\|\bv-\piw\bv\|_{L^2(\Omega_h)}^2 \lesssim
\sum_{T\in\calTh} \left( h_T^{2k+3}\|\bv\|_{H^3(\Omega)}^2
+ h_T^{2k}\|\bv\|_{H^k(\rT)}^2 \right) \lesssim
h^{2k}\|\bv\|_{H^k(\Omega)}^2.
\end{aligned}
$$
Taking square roots proves
\eqref{eq:abstract-approximation}.

\begin{remark}
The extension $\bv$ in Theorem~\ref{thm:abstract-pressure-robustness} is used only to compare functions defined on $\Omega$ and $\Omega_h$. The discrete load $\bPi_T^W \bft$ uses only $\bm f|_{\rT}$, $\Fhr$, and $D\Fhr$ for any $T \in \calTh$.  Consequently, the construction is directly implementable when the boundary parametrization and the associated mappings $\Fhr$ are available.
\end{remark}





\section{Numerical Experiments}
\label{sec:numerics}
This section presents numerical experiments to validate the theoretical results presented in this paper. The case $k=3$ is considered throughout.
The exact domain is the unit ball
$\Omega=\{(x_1,x_2,x_3)^\trans\in\mathbb{R}^3:~x_1^2+x_2^2+x_3^2<1\}.$
The load $\bft=-\nu\Delta\bu+\nabla p$ with $\nu$ specified later is determined by the exact solution
$$
\bu = \frac{1}{25}\exp(x_1+x_2+x_3)(1-x_1^2-x_2^2-x_3^2)
\begin{pmatrix}
x_3-x_2\\
x_1-x_3\\
x_2-x_1
\end{pmatrix},
\quad
p=20x_1^2(1-x_1)x_2(1-x_2)x_3.
$$
For brevity, the notation $\|\bullet\|_0$ denotes the $L^2(\Omega_h)$ norm in the tables below.

\subsection{Convergence order and divergence-free test}

With $\bft_h=\piw\bft$ and $\nu=1$, Table~\ref{tab:consistency_comparison} compares the numerical results for the corrected scheme \eqref{eq:discrete-stokes} with those for the uncorrected scheme, in which the corrected bilinear form is replaced by $a_h(\bu_h,\bv_h):= \sum_{T\in\calTh}(\nabla\bu_h,\nabla\bv_h)_ T$.
The corrected scheme achieves the theoretically optimal convergence rates:
$$
\|\bu-\bu_h\|_{0}=\mathcal{O}(h^4), \quad \|\bu- \bu_h\|_{1,h}=\mathcal{O}(h^3), \quad \|p-p_h\|_{0}=\mathcal{O}(h^3).
$$
In contrast, the uncorrected scheme exhibits suboptimal convergence rates.
Both schemes remain divergence-free due to the Piola transform used for the velocity space. The consistency correction therefore restores the optimal convergence rates while preserving this property.
\begin{table}[htbp]
\centering
\vspace{-10pt}
\setlength{\tabcolsep}{5pt} {
\begin{tabular}{@{} c *{3}{cc} c @{}}
\toprule
\multicolumn{8}{c}{\text{Corrected scheme}}\\
\midrule
$1/h$& {$\|\bu-\bu_h\|_{0}$} & {rate} 
& {$\|\bu- \bu_h\|_{1,h}$} & {rate} 
& {$\|p-p_h\|_{0}$} & {rate} 
& {$\|\Div \bu_h\|_{0}$}\\
\midrule
4  & 6.6270E$-$05 & ---  & 2.6584E$-$03 & ---  & 1.1780E$-$02 & ---  & 1.0154E$-$15 \\
6  & 1.3747E$-$05 & 3.88 & 8.1759E$-$04 & 2.91 & 3.8461E$-$03 & 2.76 & 1.4076E$-$15 \\
8  & 4.0805E$-$06 & 4.22 & 3.1479E$-$04 & 3.32 & 1.4935E$-$03 & 3.29 & 1.9085E$-$15 \\
10 & 1.6687E$-$06 & 4.01 & 1.6289E$-$04 & 2.95 & 7.9902E$-$04 & 2.80 & 2.3559E$-$15 \\
12 & 8.0549E$-$07 & 3.99 & 9.3143E$-$05 & 3.07 & 4.6860E$-$04 & 2.93 & 2.8049E$-$15 \\
\midrule
\multicolumn{8}{c}{\text{Uncorrected scheme}}\\
\midrule
$1/h$& {$\|\bu-\bu_h\|_{0}$} & {rate} 
& {$\|\bu-\bu_h\|_{1,h}$} & {rate} 
& {$\|p-p_h\|_{0}$} & {rate} 
& {$\|\Div \bu_h\|_{0}$}\\
\midrule
4  & 6.7491E$-$05 & ---  & 2.6452E$-$03 & ---  & 1.1852E$-$02 & ---  & 1.0103E$-$15 \\
6  & 1.4965E$-$05 & 3.72 & 8.3807E$-$04 & 2.83 & 3.9356E$-$03 & 2.72 & 1.4167E$-$15 \\
8  & 5.4751E$-$06 & 3.50 & 3.6631E$-$04 & 2.88 & 1.5981E$-$03 & 3.13 & 1.9178E$-$15 \\
10 & 3.0121E$-$06 & 2.68 & 2.2869E$-$04 & 2.11 & 9.1696E$-$04 & 2.49 & 2.3592E$-$15 \\
12 & 1.9453E$-$06 & 2.40 & 1.6369E$-$04 & 1.83 & 5.7956E$-$04 & 2.52 & 2.8027E$-$15 \\
\bottomrule
\end{tabular}}
\caption{Errors and convergence rates for the scheme with and without consistency corrections.}
\label{tab:consistency_comparison}
\end{table}

\subsection{Pressure-robustness test}
This subsection examines the pressure robustness of the corrected scheme \eqref{eq:discrete-stokes} by varying the viscosity $\nu$ on a fixed mesh with $h=1/4$.
Table~\ref{tab:pressure_robustness} compares the velocity errors when the discrete load $\bft_h$ in \eqref{eq:discrete-stokes} is taken as either $\piw\bft$ or $\boldsymbol{I}_h^{\mathrm{iso}}\bft$, where $\boldsymbol{I}_h^{\mathrm{iso}}\bft$ denotes the cubic nodal Lagrange interpolant of~$\bft$.
The velocity errors for $\bft_h=\piw\bft$ remain nearly unchanged as $\nu$ varies, whereas those for $\bft_h=\boldsymbol{I}_h^{\mathrm{iso}}\bft$ increase approximately like $1/\nu$. This shows that the choice $\bft_h=\piw\bft$ yields a pressure-robust velocity approximation.
\begin{table}[htbp]
\centering
\vspace{-10pt}
\setlength{\tabcolsep}{5pt} {
\begin{tabular}{@{} c cc cc @{}}
\toprule
& \multicolumn{2}{c}{$\bft_h=\piw\bft$}
& \multicolumn{2}{c}{$\bft_h=\boldsymbol{I}_h^{\mathrm{iso}}\bft$} \\
\cmidrule(lr){2-3}\cmidrule(lr){4-5}
$\nu$
& $\|\bu-\bu_h\|_0$
& $\|\bu-\bu_h\|_{1,h}$
& $\|\bu-\bu_h\|_0$
& $\|\bu-\bu_h\|_{1,h}$ \\
\midrule
$10^{-1}$ & 6.6270E$-$05 & 2.6576E$-$03 & 7.2434E$-$05 & 2.6800E$-$03 \\
$10^{-3}$ & 6.6270E$-$05 & 2.6576E$-$03 & 2.8166E$-$03 & 3.3629E$-$02 \\
$10^{-5}$ & 6.6270E$-$05 & 2.6576E$-$03 & 2.8153E$-$01 & 3.3518E$+$00 \\
$10^{-7}$ & 7.7225E$-$05 & 2.7097E$-$03 & 2.8153E$+$01 & 3.3518E$+$02 \\
\bottomrule
\end{tabular}}
\caption{Velocity errors on a fixed mesh with $h=1/4$ for varying $\nu$ and two choices of $\bft_h$ in \eqref{eq:discrete-stokes}.}
\label{tab:pressure_robustness}
\end{table}

\bibliographystyle{siamplain}
\bibliography{evp}
\end{document}